\documentclass[10pt]{amsart}
\usepackage{latexsym, amsmath, amsfonts, amssymb, mathrsfs, fancyhdr, tikz, tikz-cd}
\usepackage{verbatim}
\usepackage{multicol}
\usepackage{enumitem} 
\usetikzlibrary{matrix,arrows,decorations.pathmorphing,calc}

\newtheorem{theorem}{Theorem}[section]
\newtheorem{lemma}[theorem]{Lemma}
\newtheorem{cor}[theorem]{Corollary}

\newtheorem*{thm}{Main Theorem}
\newtheorem{proposition}[theorem]{Proposition}

\theoremstyle{definition}
\newtheorem{definition}[theorem]{Definition}
\newtheorem{example}[theorem]{Example}

\theoremstyle{remark}
\newtheorem{remark}[theorem]{Remark}

\numberwithin{equation}{section}

\usepackage[left=1.45in, right=1.45in]{geometry}

\DeclareMathOperator{\ind}{ind}

\newcommand{\Sym}{{\text {Sym}}}

\begin{document}

\title{Small Cancellation for Random Branched Covers of Groups}

\author[H.~Cho]{Hyeran Cho}
\address{Tufts University, 177 College Ave, Medford, MA 02155}
\email{hcho13@tufts.edu}

\author[J.~Lafont]{Jean-Fran\c{c}ois Lafont}
\address{The Ohio State University, 231 W. 18th Ave., Columbus OH 43210}
\email{lafont.1@osu.edu}

\author[R.~Skipper]{Rachel Skipper}
\address{University of Utah, 155 S 1400 E, Salt Lake City, UT 84112}
\email{rachel.skipper@utah.edu}


\date{\today.}

\begin{abstract}
We construct a random model for an $n$-fold branched cover of a finite acceptable $2$-complex $X$. This includes presentation $2$-complexes for finitely presented groups satisfying some mild conditions. For any $\lambda >0$, we show that as $n$ goes to infinity, a random branched cover asymptotically almost surely is homotopy equivalent to a $2$-complex satisfying geometric small cancellation $C'(\lambda)$. As a consequence the fundamental group of a random branched cover is asymptotically almost surely Gromov hyperbolic and has small cohomological dimension.
\end{abstract}

\maketitle

\section{Introduction}\label{Section:Introduction}

The probabilistic method was originally pioneered by Erd\"os, and was used as a non-constructive approach to showing the existence of interesting examples in combinatorics and graph theory, see e.g. the classic text by Alon—Spencer \cite{alon-spencer}. The method has been particularly effective in graph theory, with random graphs having evolved into its own field of study. Random models have since been developed to study higher dimensional simplicial complexes (Kahle \cite{Kahle2014}), random closed surfaces (Brooks—Makover \cite{Brooks-Makover}), and random $3$-manifolds (Dunfield—Thurston \cite{Dunfield-Thurston}).  In the late 1980s, Gromov launched the study of random groups. The two main models for random groups are the density model and the few relator model, see Gromov \cite{Gromov-asymptotic-invariants} and the survey article by Ollivier \cite{Ollivier-2005}. 

A common theme in these approaches is that a space is built by attaching spaces together via a random process. In the case of random graphs or simplicial complexes, edges or simplices are added to a vertex set at random. In the setting of random surfaces or $3$-manifolds, triangles or handlebodies are glued together via a suitable random process. In the setting of random groups, one can think of relators as $2$-cells being randomly attached to a bouquet of circles, with the random group being the fundamental group of the resulting $2$-complex.

A different topological construction that is commonly used in low-dimensional topology is that of branched covers. All closed oriented surfaces can be realized as branched covers over the sphere, a fact that is also true for $3$-manifolds (Hilden \cite{hilden} and Montesinos \cite{montesinos}). Branched covers have also been a source of many interesting examples in the geometry of negatively curved manifolds, see e.g. Gromov—Thurston \cite{gromov-thurston}, Fine—Premoselli \cite{fine-premoselli}, Stover—Toledo \cite{stover-toledo}, Minemyer \cite{Min25}, Guenancia-Hamenst\"{a}dt \cite{GH25}. From this viewpoint, it is natural to look for a random model for branched covers. In the present paper, we construct a random model for branched covers of finite polygonal $2$-complexes and prove:

\begin{thm}\label{thm:main theorem}
    Let $X$ be an acceptable finite polygonal $2$-complex, and $X(\sigma)$ denote an $n$-fold random branched cover of $X$. Then for any fixed $\lambda >0$, the branched cover $X(\sigma)$ is asymptotically almost surely homotopy equivalent to a $2$-complex satisfying geometric $C'(\lambda)$-small cancellation.
\end{thm}

By a {\it polygonal} $2$-complex, we mean a $2$-dimensional CW-complex where the attaching maps are particularly simple, as described here. The $1$-skeleton is metrized by having each edge of length one and having a prescribed orientation. The $2$-cells are identified with disks scaled so that their perimeter is an integer. We then subdivide the boundary of the $2$-cell into consecutive intervals of length one, and the attaching map is required to map each of these intervals isometrically onto a single edge of the $1$-skeleton. For these CW-complexes, the attaching maps for each disk can be described just be enumerating the sequence of edges and their orientation traversed in the $1$-skeleton. Note that we are allowing the possibility of a disk attaching to a single edge loop, or to a pair of edges (so paths of combinatorial length one or two).

By an {\it acceptable} polygonal $2$-complex, we mean one which satisfies the following additional conditions:
\begin{enumerate}
\item the $1$-skeleton $X^{(1)}$ has fundamental group of rank at least two;
\item the $2$-cells have attaching maps that are not proper powers in $\pi_1(X^{(1)})$ (so in particular, are non-trivial), and that are pairwise non-homotopic.
\end{enumerate}
It is easy to see that the Main Theorem {\bf fails}\footnote{ More precisely, the theorem holds trivially when the rank is one (see Remark \ref{rmk-one-generator}), and fails when two attaching maps are homotopic (these create $\pi_2$, see Lemma \ref{lem:pi2}). In the presence of proper powers, a positive proportion of the branched covers produced will have non-trivial $\pi_2$, see Section \ref{subsec:non-acceptable}.} 
for non-acceptable $2$-complexes, so our hypotheses are actually necessary.

It is well-known that the $C'(1/6)$-small cancellation property has strong geoemetric consequences, e.g. \cite{gromov2003random}, \cite{wise-small-cancellation}. As a result, the $\lambda=1/6$ case of our main theorem immediately implies the following:

\begin{cor}
    Let $X$ be an acceptable finite $2$-complex, and let $X(\sigma)$ be a $n$-fold random branched cover of $X$. Then we have asymptotically almost surely the following properties hold:
    \begin{itemize}
        \item $X(\sigma)$ is aspherical, hence an Eilenberg--MacLane space $K(\pi_1(X(\sigma)),1)$;\item $\pi_1(X(\sigma))$ is Gromov hyperbolic and cubulable;
        \item $\pi_1(X(\sigma))$ is torsion-free, and has cohomological dimension $\leq 2$
    \end{itemize}
\end{cor}

It is perhaps worth emphasizing that our model is qualitatively very different from models previously considered in the literature. Let us highlight a few of the key differences:
\begin{itemize}
    \item in our model, the number of generators for the random groups is growing to infinity as $n$ increases. In previous models, the number of generators was always fixed.
    \item in our model, the groups that are produced {\it virtually surject} to the initial group. In previous models, the groups produced are typically quotients of the initial group.
    \item in our model, the key probabilistic ingredient are asymptotic properties of random elements in finite symmetric groups. These do not appear to feature in any of the previous models. 
\end{itemize}
Thus while the statements of our results are similar to those known for other models, the methods and context are quite different.

In Section 2, we will briefly review some basic notions related to branched covers and to small cancellation. In Section 3 we will describe our model for random branched coverings.

\subsection*{Acknowledgments} Cho and Lafont were partially supported by the NSF, under grant DMS-2407438. Skipper was partially supported by the NSF, under grants DMS-2343739 and DMS-2506840, and the European Research Council (ERC), under the European Union’s Horizon 2020 research and innovation program (grant agreement No.725773).
The authors would like to thank the following colleagues for insightful conversations and/or helpful comments: Laurent Bartholdi, Francesco Fournier-Facio, Jingyin Huang, Matthew Kahle, Ilya Kapovich, Olga Kharlampovich, MurphyKate Montee, Denis Osin, Tim Riley, Genevieve Walsh, and Matthew Zaremsky. The authors would also like to thank Doron Puder for comments on an earlier draft that helped to unify the main construction.

\section{Preliminaries}\label{Section:Prelim}

\subsection{Branched Coverings}

Let us recall the notion of branched covering of smooth manifolds (see e.g. \cite{gompf20234}). 

\begin{definition}\textbf{($n$-fold branched covering of manifolds)}\label{def-branched-cover}
    Given a pair of smooth $k$-manifolds $X^k$, $Y^k$, an $n$-fold branched (or ramified) covering is a smooth, proper map $f:X^k\rightarrow Y^k$ exhibiting some particularly simple local form. The critical set $B^{k-2}\subset Y$ is called the branch locus, and we require that it is a smoothly embedded codimension two submanifold. Moreover, $f|_{X\setminus f^{-1}(B)}:X\setminus f^{-1}(B)\rightarrow Y\setminus B$ is a covering map of degree $n$, and for each $p\in f^{-1}(B)$ there are local coordinate charts $U,V\rightarrow\mathbb{C}\times\mathbb{R}^{k-2}$ about $p,f(p)$ on which $f$ is given by $(z,x)\mapsto (z^m,x)$ for some positive integer $m$ called the branching index of $f$ at $p$. 
\end{definition}

Notice that, when restricted to the branching locus $B$, the map $f|_{f^{-1}(B)}$ is just an
ordinary covering map. The pre-image of $B$ is not assumed to be connected, and indeed, could have multiple connected components. Transverse to the branching locus $B$, $f$ behaves like the map $z\rightarrow z^m$ near the origin -- though again, at different pre-image points the value of $m$ might be different.

This definition can be readily extended to the setting of CW-complexes. As we only need the $2$-dimensional case, we will focus on that setting. 

\begin{definition}\textbf{($n$-fold branched covering of $2$-complexes)}
    Given a pair of finite $2$-dimensional CW-complexes $X$, $Y$, an $n$-fold branched (or ramified) covering is a continuous map $f:X\rightarrow Y$ satisfying the following property. There is a finite subset of points $B\subset Y$, called the branching locus, which satisfies $B \cap Y^{(1)}=\emptyset$ (so $B$ lies in the interior of the $2$-cells). Moreover, $f|_{X\setminus f^{-1}(B)}:X\setminus f^{-1}(B)\rightarrow Y\setminus B$ is a covering map of degree $n$, and for each $p\in f^{-1}(B)$ there are local coordinate charts $U,V\rightarrow\mathbb{C}$ about $p,f(p)$ on which $f$ is given by $z\mapsto z^m$ for some positive integer $m$, called the branching index of $f$ at $p$. 
\end{definition}

Note that, in the case where a $2$-cell in $Y$ contains more than one branch point, connected components of its preimage might no longer be homeomorphic to a disk. However, if there is a single branch point inside a $2$-cell, then each component in its preimage will be homeomorphic to a disk.

\vskip 10pt

\subsection{Small Cancellation Conditions}

Small cancellation has been a useful tool in combinatorial group theory since the 1970s, see e.g. the references \cite{lyndon1977combinatorial}, \cite{gromov1987hyperbolic}, \cite{gromov2003random}, and \cite{guirardel2012geometric}.
There are various notions of small cancellation. Here, we start
by recalling the  classical small cancellation with respect to a group presentation. Roughly speaking, this condition says that any common subword between two relators in a presentation is short compared to the length of the relators.

Let $X$ be a symmetric generating set for a group $\Gamma$, i.e. $X$ contains all elements of a generating set $S$ and their inverses. We call an element of $S$ a $\textit{letter}$. A $\textit{word}$ $w$ is finite string of letters $w=s_1\dots s_m$. We consider $w$ as an element of the free group $F$ with the generating set $S$. Then each element of $F$ other than the identity $1$ has a unique representation as a $\textit{reduced word}$ $w=s_1\dots s_n$ in which no two successive letters $s_i s_j$ form an inverse pair $s_i s_i^{-1}$. The integer $n$ is the $\textit{length}$ of $w$, which we denote by $|w|$. A reduced word is called $\textit{cyclically reduced}$ if $s_n$ is not the inverse of $s_1$. If there is no cancellation in the product $z=u_1\dots u_n$, we write $z\equiv u_1\dots u_n$.

A subset $R$ of $F$ is called $\textit{symmetrized}$ if all elements of $R$ are cyclically reduced and for each $r$ in $R$, all cyclically reduced cyclic permutations of both $r$ and $r^{-1}$ also belong to $R$. Suppose that $r_1\equiv bc_1$ and $r_2\equiv bc_2$ are distinct elements of $R$. If $b$ is the maximal such subword then it is called a $\textit{piece relative to the set R}$ or simply a $\textit{piece}$.

\begin{definition}
   We say that $R$ satisfies the small cancellation condition $C'(\lambda)$ if for $r\in R$ with $r\equiv bc$ where $b$ is a piece, then $|b|<\lambda |r|$. In this case, we also say that the presentation satisfies $C'(\lambda)$. Also, for a group $\Gamma$, if there is a presentation that satisfies $C'(\lambda)$, we say that $\Gamma$ is $C'(\lambda)$ group.
\end{definition}

The following is well known:
\begin{proposition}\label{gromov}\cite{gromov1987hyperbolic}
    If a finitely presented group $\Gamma$ satisfies $C'(\frac{1}{6})$, then $\Gamma$ is word hyperbolic. 
\end{proposition}

Geometric consequences of the small cancellation hypothesis are further studied in \cite{lyndon1977combinatorial}, typically via the group's presentation $2$-complex, as well as van Kampen diagrams and their mapping to $2$-complexes. This allows the small cancellation condition to be reformulated geometrically, and the results to be generalized to the setting of polygonal $2$-complexes. 

Recall that the attaching maps for the $2$-cells in a polygonal $2$-complexes are given by a (cyclic) sequence of directed edges from the $1$-skeleton. For each $2$-cell $D$ in a polygonal $2$-complex, the boundary $\partial D$ is a cycle graph, and one can label the edges of $\partial D$ according to the directed edges they map to in the $1$-skeleton.

 We will consider combinatorial subpaths $b$ in $\partial D$ which are injective on their interior (so at most, agree at the two endpoints). Note that we allow the two extremal cases (i) where the path has length zero (so is just a point), and (ii) where the path covers the entirety of $\partial D$. A {\it subpiece} is a subpath in the boundary of a pair of disks $\partial D$, $\partial D'$, whose labels, including orientation, are identical. Note that $D'$ can possibly be the same disk $D$, but with distinct initial vertices for the two subpaths of the boundary. In that case a subpiece would be contained in the self-intersection of the attaching map on $\partial D$. A \textit{piece} is a subpiece in the boundary of a pair of disks $\partial D, \partial D'$ which is maximal under containment. 
 
\begin{remark}
    Maximality gives us some insight on the local behavior of the labels on $D$, $D'$ near the endpoints. If the endpoints of the piece are distinct vertices in $\partial D$, $\partial D'$, then maximality tells us that at the initial (and terminal) vertex of the path $b$, the previous (resp. following) edges of $\partial D$ and $\partial D'$ have to be distinct. 
    
    The other possibility is that the endpoints of the piece $b$ coincide in $D$ (for example). Note that this means the label on $D'$ contains an entire copy of the label on $D$.
\end{remark} 

We say that $D$ satisfies $C'(\lambda)$ if for any piece $b$ of $D$, we have $$\frac{\ell(b)}{\ell(\partial D)}<\lambda$$
where $\ell$ is the combinatorial path length. We say that a $2$-complex satisfies $C'(\lambda)$ if each of its $2$-cells satisfy $C'(\lambda)$. This provides us with a geometric notion of small cancellation, and results on groups satisfying classical small cancellation (established via analysis of the presentation $2$-complex) readily generalize to $2$-complexes satisfying geometric small cancellation. 

\medskip

 In our later constructions, we will consider certain finite covers of the $1$-skeleton of $X$, with certain lifts of attaching maps. Given a $2$-cell $D$ with attaching map $\alpha: \partial D \rightarrow X^{(1)}$, we have an associated map $\tilde \alpha: \mathbb R \rightarrow X^{(1)}$ obtained by composing the universal covering map $\pi: \widetilde{\partial D} \rightarrow \partial D$ with the attaching map. We can identify $\widetilde{\partial D}$ with $\mathbb R$ equipped with its standard simplicial structure. The map $\tilde \alpha$ is then described by the bi-infinite, periodic word obtained by lifting the edge labels from $\partial D$ to $\widetilde{\partial D} \cong \mathbb R$.  Since we will be interested in studying small cancellation properties associated to some of these covers, we now formulate a notion that is slightly more general than a piece.

A {\it sub-overlap} between two disks $D$, $D'$ is a pair of finite combinatorial subpaths $\bf{p}\subset \widetilde{\partial D}$ and $\bf{p'}\subset \widetilde{\partial D'}$ on which the lifted attaching maps $\tilde \alpha: \widetilde{\partial D} \rightarrow X^{(1)}$ and $\tilde \beta: \widetilde{\partial D'} \rightarrow X^{(1)}$ coincide. Note that the paths $\bf{p}, \bf{p'}$ are not required to respect the chosen orientations on $\widetilde{\partial D}, \widetilde{\partial D'}$. Sub-overlaps are considered equivalent if they differ by translation by the $\pi_1(\partial D)$ and $\pi_1(\partial D')$ actions. We also allow the case where $D=D'$, in which case we also require the starting point of the overlaps $\bf{p} \subset \widetilde{\partial D}$ and $\bf{p'} \subset \widetilde{\partial D'}$ to be in distinct orbits of the $\pi_1(\partial D)$-action (i.e. correspond to distinct initial points in $\partial D$) if both paths $\bf{p}, \bf{p'}$ have the same relative orientation. When $D=D'$ and the paths $\bf{p}, \bf{p'}$ have opposite orientations (relative to the fixed orientation on $\widetilde{\partial D}$) they will always be considered distinct. 

\begin{definition}
       An {\it overlap} is a sub-overlap which is maximal under containment. The {\it overlap ratio} of $D$ with $D'$ is defined to be 
       $$o(D, D') = \sup_{\bf{p}} \frac{\ell(\bf{p})}{\ell(\partial D)}$$
       where the supremum is over all overlaps $(\bf{p},\bf{p'})$ between $D$ and $D'$. The overlap ratio of a $2$-cell $D$ is then defined to be $o(D):= \sup_{D'} o(D, D')$, and the overlap ratio of the polygonal $2$-complex $X$ is defined by $o(X) = \sup_D o(D) = \sup_{D, D'} o(D, D')$. 
 \end{definition}

Observe that any piece automatically gives a sub-overlap. In particular, for $\epsilon <1$ the $C'(\epsilon)$ small cancellation condition is implied by the statement that  $o(X)<\epsilon$. 
On the other hand, when the overlap ratio of a pair satisfies $o(D, D')\geq 1$, this just means that the label for $\partial D$ is entirely contained in the label for $\partial D'$.

\begin{lemma}\label{lem-bound-on-overlap-length}
    Let $X$ be an acceptable polygonal $2$-complex, $(\bf{p},\bf{p'})$ an overlap between disks $D, D'$, and set $M=\max \{\ell(\partial D), \ell(\partial D')\}$. Then the length of the overlap is bounded above by $\ell({\bf p}) < M^2+M$.
    In particular, for any pair $D, D'$ of $2$-cells, the overlap ratio $o(D, D')$ is finite. It follows that for any finite acceptable polygonal $2$-complex, the overlap ratio $o(X)$ is finite.
\end{lemma}

\begin{proof}
    Let us consider the case where $D\neq D'$, and assume that the overlap $(\bf{p},\bf{p'})$ has length $$\ell({\bf p}) \geq M^2 + M.$$ Since the lifted attaching maps $\bar \alpha, \bar \beta$ coincide on $\bf{p}, \bf{p'}$, these subpaths of $\widetilde{\partial D}, \widetilde{\partial D'}$ have identical edge labelings. The labeled bi-infinite paths $\widetilde{\partial D}, \widetilde{\partial D'}$ are periodic with respect to the $\pi_1(\partial D)$-action and $\pi_1(\partial D')$-action, which are translations by $\ell(D), \ell(D')$ respectively. Since the common subpath contains fundamental domain for both translations, one can apply the Euclidean algorithm to find a subpath of length $r=GCD(\ell(D), \ell(D'))$ that tiles both fundamental domains. To see this, consider the initial subpath ${\bf q} \subset {\bf p}$ of length $r$. Since $r=GCD(\ell(D), \ell(D'))$, the Euclidean algorithm provides us with an integral solution to B\'ezout's identity
    $$r=A\ell(D) + B\ell(D')$$
    and the solution satisfies $|A|\leq \ell (D'), |B|\leq \ell(D)$. Note that exactly one of the integers $A,B$ is positive, the other is negative. Assume without loss of generality that $A$ is positive. Viewing $\bf{q} \subset \bf{p'}$ as the initial segment of $\bf{p'}$, and using $\pi_1(D')$-periodicity of $\widetilde{\partial D'}$, we can translate $A$ times along $\widetilde{\partial D'}$. Since the length of ${\bf p'}$ is at least $M^2+M$, this translate of ${\bf q}$ is still contained within ${\bf p'}$. We now switch to viewing that translate as contained in ${\bf p}$, and use $\pi_1(D)$-periodicity of $\widetilde{\partial D}$ to translate $B$ times along $\widetilde{\partial D}$. This has the effect of translating $\bf{q}$ by exactly $r$, and hence the initial portion of ${\bf p}$ of length $2r$ consists of two copies of $\bf{q}$. We can iterate this process $M/r$-times, noting that the hypothesis that $\ell({\bf p})\geq M^2+M$ guarantees that the forward and backward translates from B\'ezout's identity land within the common subwords ${\bf p}, {\bf p'}$. If $r<M$, this tells us that the larger of the two words is a proper power, contradicting the definition of acceptable $2$-complex. On the other hand, if $r=M$ we get that $\ell(D)=\ell(D')$, and the two disks $D, D'$ are attached along the same map, which again contradicts the  definition of acceptable $2$-complex.

    Next we consider the case where $D=D'$, i.e. self-overlaps. Then one has that the bi-infinite label on $\widetilde{\partial D}$ is periodic with respect a translation by $\ell(D)$. Since the subwords ${\bf p}, {\bf p'}$ differ by a translation by some $0<k<\ell(D)$, the subword ${\bf p}$ is also $k$-periodic. Then as before we can apply the Euclidean algorithm to obtain a solution to B\'ezout's identity, and use combinations of $\ell(D)$-translations and $k$-translations to see that the ${\bf p}$ is actually periodic with period $r=GCD(k, \ell(D)) < \ell(D)$. This implies the attaching map for $D$ is a proper power, which again contradicts $X$ acceptable. 
 \end{proof}

\begin{cor}\label{cor-bound-on-number-of-overlaps}
    Any acceptable polygonal $2$-complex $X$ only contains finitely many equivalence classes of overlaps $(\bf{p},\bf{p'})$.
\end{cor}

\begin{proof}
    For a given pair of $2$-cells $D$, $D'$, we can count the overlap pairs $(\bf{p},\bf{p'})$. Up to the action of $\pi_1(\partial D)$, the initial point of the path $\bf{p}$ can be chosen in a fixed fundamental domain $F\subset \widetilde{\partial D}$, where $F$ is a combinatorial interval of length $\ell(\partial D)$. Thus there are at most $\ell(\partial D)$ possible initial vertices for $\bf{p}$. From Lemma~\ref{lem-bound-on-overlap-length}, there is also a uniform bound on the length of $\bf{p}$. Thus there are at most finitely many possibilities for the path $\bf{p}$. A symmetric argument shows that there are at most finitely many choices for $\bf{p'}$, hence finitely many possibilities for the pairs $(\bf{p},\bf{p'})$. Since the acceptable $2$-complex $X$ has a finite number of $2$-cells, the corollary follows.
\end{proof}

\bigskip

\section{Random Branched Coverings} \label{Section:Branched}

In this section, we describe our random model for branched coverings of finite $2$-complexes, with the motivating example being the case of the presentation $2$-complex of a finitely presented group. We then establish a few basic properties concerning the behavior of $2$-cells in our random model.

 \vskip 10pt

\subsection{Branched Coverings of presentation $2$-complexes}\label{subsec:brcov}

Let us first focus on the setting of a presentation $2$-complex. Let $\Gamma=\langle u_1,\dots,u_t\mid r_1,\dots,r_s\rangle$ be a finite presentation of a group $\Gamma$, and let $X$ be the presentation $2$-complex for the fixed group presentation above. The complex consists of a single vertex $v$ along with oriented loops $x_1,\dots, x_t$ corresponding to each generator, and $2$-disks $D_1,\dots, D_s$ that are attached by the attaching maps $r_1,\dots,r_s$ corresponding to the relators. For such complexes, being an acceptable $2$-complex just means that the finitely presented group has at least two generators, that no relator is conjugate to a proper power, and that no relations are conjugate to each other.

\begin{remark}
    Any finite polygonal $2$-complex that has a single vertex can be viewed as a presentation $2$-complex for its fundamental group. These spaces have the advantage of having a canonical basepoint for the fundamental group. To deal with the general case of a finite connected polygonal $2$-complex $X$, we can contract a spanning tree $T$ for the $1$-skeleton to obtain a $1$-vertex $2$-complex $X'$, and use the branched cover model for covers of $X'$. The details can be found in Section~\ref{sec:multiple-vertex-case}.
\end{remark}

Next let us analyze a branched cover of $\rho:\hat X \rightarrow X$ between polygonal $2$-complexes from the viewpoint of the attaching maps.
The branching locus $B \subset X$ consists of the set of centers of the $2$-disks $D_1,\dots, D_s$. On the level of the $1$-skeletons, we have a not necessarily connected $n$-fold covering of the $1$-skeleton $X^{(1)}$. Aside from the covering $\hat X^{(1)}$ of $X^{(1)}$, the branched cover $\hat X$ has a number of $2$-disk attached to the $1$-skeleton. The covering condition on the $1$-skeleton forces the attaching maps of these disks to follow the lifts of the attaching map $r_1,\dots,r_s$.

\begin{remark}[Labeling convention]\label{rmk:label-convention}
    Since the presentation $2$-complex $X$ has a single vertex, an $n$-fold covering of $X^{(1)}$ has $n$ vertices, which we will label $v_1$ to $v_n$. Under the covering map, the pre-image of each directed loop $x_i$ will consist of $n$ directed edges which we denote $x_{i1},\dots,x_{in}$. Our labeling convention is to label the lifted edge $x_{ij}$ to originate at the lifted vertex $v_j$. Then each vertex $v_j$ has some lifted edge  $x_{ik}$ coming in, and the lifted edge $x_{ij}$ going out. Note the situation where $j=k$ corresponds to the case where the lifted edge starting at $v_j$ is a loop at the vertex. 
\end{remark}

  \begin{figure}[h]\label{fig branched covering}
\includegraphics[width=7cm,height=5cm]{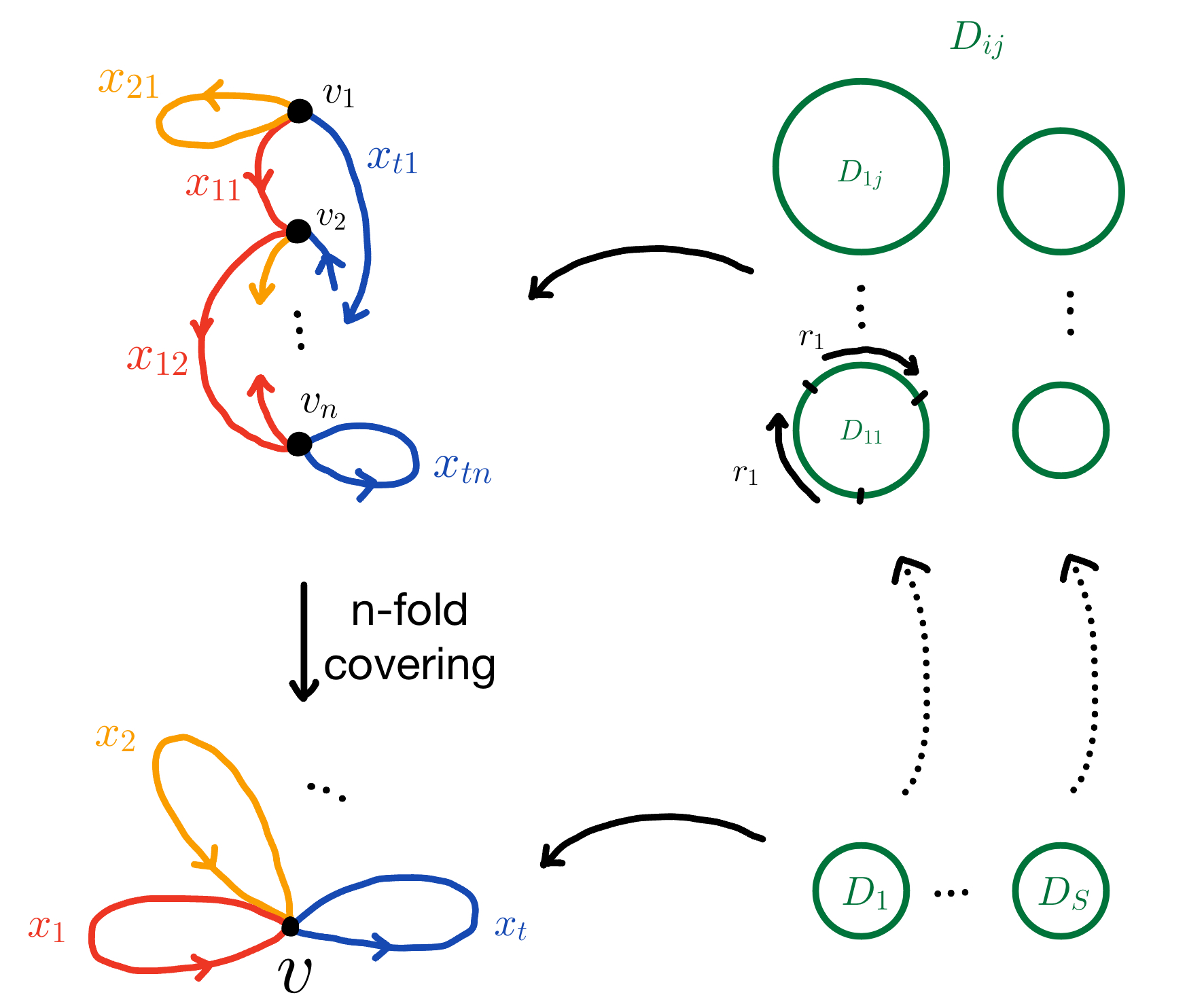}
\caption{A branched cover of a presentation $2$-complex}
\label{fig:branchedcoverpres}
\end{figure}

This now suggests a method for producing branched covers. Start with an ordinary finite cover of the $1$-skeleton $\rho: \widehat{X^{(1)}} \rightarrow X^{(1)}$. Note that each disk $D_i$ in $X$ has its boundary labeled $r_i$, and for each vertex in the cover we have a path starting at that vertex and following the letters of $r_i$. If this path in the cover ends at a vertex different from the starting vertex, we follow the letters of $r_i$ again, repeating the process until the word ends at the starting vertex. In this process, we get a closed loop in $\widehat{X^{(1)}}$, and we can attach a $2$-disk to $\widehat{X^{(1)}}$ along this lifted loop. We call these disks $D_{ij}$, for suitable indexing set for $j$, and refer to $D_{ij}$ as a \emph{lift} of the disk $D_i$. Note that some lifted loops coincide up to cyclic permutation, due to the choice of where to start a lift. We only attach a single disk to $\widehat{X^{(1)}}$ along each of these loops. The degree of the covering map on the boundary of the disks $\partial D_{ij} \rightarrow \partial D_i$ is denoted by $\ind(D_{ij})$ and called the \emph{index} of the disk $D_{ij}$. Attaching all lifts of disks $D_1,\dots,D_s\subset X$, we obtain a space $\widehat{X}$ and a candidate branched covering map $\widehat{X}\rightarrow X$, where the branching locus is the set of the centers $c_{i}\subset D_{i}$ of all of the disks in the original 2-complex. See Figure~\ref{fig:branchedcoverpres} for an illustration of this process.

\begin{remark}
    The branching index of each center $c_{ij}\subset D_{ij}$ is the same as the index of the disk $D_{ij}$. The combinatorial lengths of the attaching maps are related via the simple formula $$\ell(\partial D_{ij})=ind(D_{ij})\ell(\partial D_i)=ind(D_{ij})|r_i|$$
    where $\ell$ is the combinatorial length.
\end{remark} 

\vskip 10pt

\subsection{Random model for branched coverings}\label{subsec:random-model}

We now proceed to define our random model for branched coverings of a presentation $2$-complex, which allows us to 
randomly pick a degree $n$ branched cover of the $2$-complex $X$. We call this the  {\it random labeled branched cover model}. Note that, as detailed in the previous section, each degree $n$ branched cover of the $2$-complex $X$ determines, and is determined by, an ordinary degree $n$ cover of the $1$-skeleton $X^{(1)}$. Since in our special case $X^{(1)}$ is a bouquet of $t$ circles, it is easy to describe the degree $n$ covers of this graph.

Labeling the $n$ pre-images of the single vertex $v$ by labels $V=\{v_1, \ldots, v_n\}$, covering space theory tells us that associated to each loop $x_i$ in $X$, we have a permutation $\sigma_i$ of the vertex set $V$. The collection of permutations is determined by the finite cover, and conversely, determines the cover up to label preserving isomorphism. Thus there is a bijection between the set of degree $n$ labeled branched covers, and elements in the product of $t$ copies of the symmetric group $\Sym(n)$.

A candidate random $n$-fold covering of $X^{(1)}$ can now be generated by choosing $t$ random permutations $\sigma=(\sigma_1,\dots,\sigma_t)$ with uniform distribution, where each $\sigma_i\in \Sym(n)$ corresponds to each generator $u_i$ for $i=1,\dots,t$. For a generator $u_i$ and its corresponding loop $x_i$ in $X^{(1)}$, a permutation $\sigma_i$ encodes the preimage of $x_i$ in the $n$-fold covering space. More precisely, if the permutation $\sigma_i$ maps the integer $a$ to the integer $b$, then there exists an oriented pre-image of the edge $x_i$ that joins the vertex $v_a$ to the vertex $v_b$. For a finite presentation of a group $\Gamma$, the random choice of $\sigma=(\sigma_1,\dots,\sigma_t)$ where $\sigma_i\in \Sym(n)$ completely determines a $n$-fold covering of $X^{(1)}$, and thus we can carry out the procedure described in the previous section, and attach $2$-disks along all lifts of the attaching maps. We denote the resulting space by $X(\sigma)$. Let us look at a few examples.

\begin{figure}[h!]
\includegraphics[width=10cm,height=5cm]{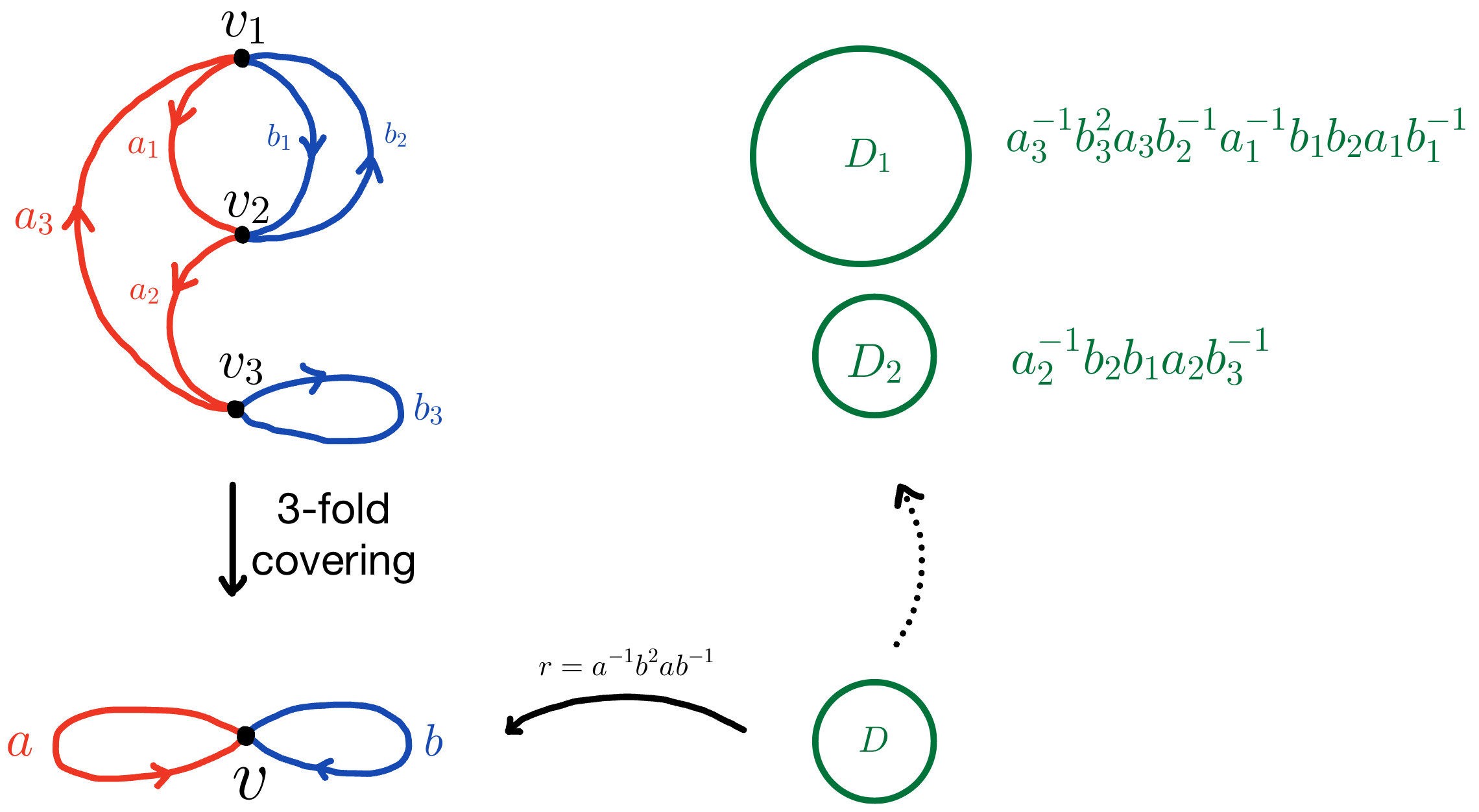}
\caption{A branched covering of degree $3$ (Example \ref{ex})}\label{fig_ex1}
\end{figure}

\begin{example}\label{ex}
    Let $\Gamma=\langle a,b \mid a^{-1}b^2ab^{-1}\rangle$ and $X$ be its presentation $2$-complex. Let $n=3$ and $\sigma_a=(123),\sigma_b=(12)(3)\in \Sym(3)$. From the choice of $\sigma=(\sigma_a,\sigma_b)$, we have a $3$-fold covering of $X^{(1)}$. A disk $D\subset X$ is attached to $X^{(1)}$ along the relator $a^{-1}b^2ab^{-1}$, so our label on $\partial D$ is $a^{-1}b^2ab^{-1}$. The attaching map of $D$ has two connected lifts $D_1$ and $D_2$, which are attached via the maps $$\partial{D_1}=a_3^{-1}b_3^2a_3b_2^{-1}a_1^{-1}b_1b_2a_1b_1^{-1},$$ 
    $$\partial{D_2}=a_2^{-1}b_2b_1a_2b_3^{-1}.$$ 
    Then the obtained $2$-complex $X(\sigma)$ would be a $3$-fold branched covering of $X$ where $\{c_1\}\cup\{c_2\}$ is the preimage of the branching locus with $c_1$ (resp. $c_2$) the center of the disk $D_1$ (resp. $D_2$). The branching index at $c_1$ is $2$ and at $c_2$ is $1$. 
    An illustration of the cover $X(\sigma)$ is given in Figure~\ref{fig_ex1}.

    Consider the fundamental group $\pi_1(X(\sigma))$ in this case. Choose $a_1,a_2$ as a spanning tree of the $1$-skeleton. Then the generators of $\pi_1(X(\sigma))$ will be $a_3,b_1,b_2$ and $b_3$. We have two relators $a_3^{-1}b_3^2a_3b_2^{-1}b_1b_2b^{-1}$ and $b_2b_1b_3^{-1}$, obtained by collapsing the spanning tree $a_1\cup a_2$. We can remove the generator $b_3$ and the relator $b_2b_1b_3^{-1}$ by a Tietze transformation as $b_3=b_2b_1$. Therefore, the fundamental group of the branched cover $X(\sigma)$ is the one relator group with presentation
    \begin{align}
        \pi_1(X(\sigma))\cong \langle a_3,b_1,b_2\mid a_3^{-1}b_2b_1b_2b_1a_3b_2^{-1}b_1b_2b_1^{-1}\rangle.
    \end{align}

\end{example}

\bigskip

\begin{figure}[h!]
\includegraphics[width=8cm,height=5.33cm]{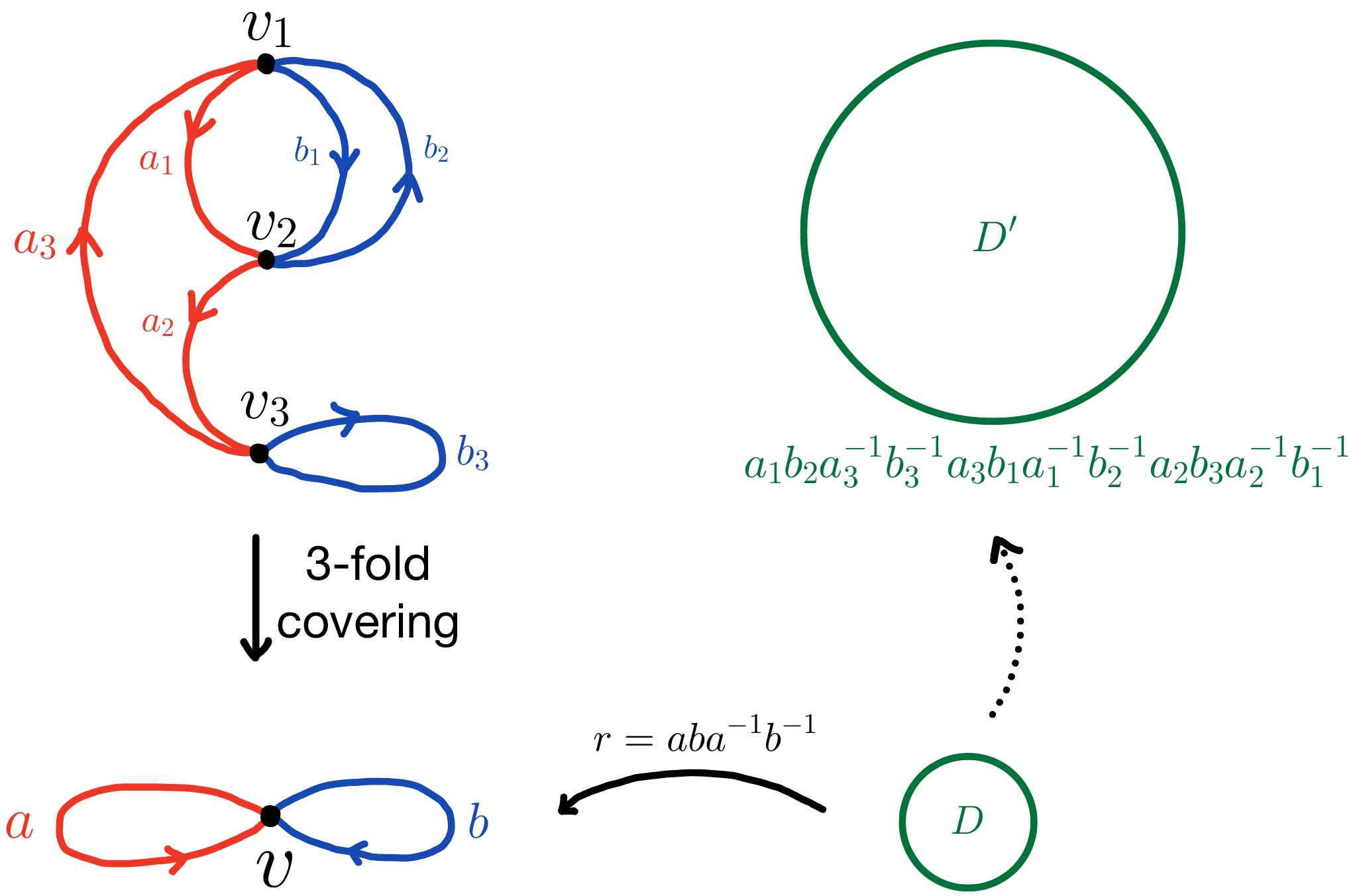}
\caption{A branched covering of degree $3$ (Example \ref{ex2})}\label{fig_ex2}
\end{figure}

\begin{example}\label{ex2}
Let $\Gamma=\langle a,b\mid aba^{-1}b^{-1}\rangle$ and $X$ be its presentation $2$-complex. Let $n=3$ and as before, let $\sigma_a=(123),\sigma_b=(12)(3)\in \Sym(3)$. From the choice of $\sigma=(\sigma_a,\sigma_b)$, we have a $3$-fold covering of $X^{(1)}$. A disk $D\subset X$ is attached to $X^{(1)}$ along the relation $aba^{-1}b^{-1}$ so the label on $\partial D$ is $aba^{-1}b^{-1}$. In this case the attaching map of $D$ has a unique connected lift $D'$, with label $$\partial{D'}=a_1b_2a_3^{-1}b_3^{-1}a_3b_1a_1^{-1}b_2^{-1}a_2b_3a_2^{-1}b_1^{-1},$$ see Figure~\ref{fig_ex2}. Thus the branched cover $2$-complex $X(\sigma)$ will be a $3$-fold branched covering of $X$. The preimage of the branched point is the unique point $\{c'\}$ where $c'$ is the center of the disk $D'$.  The branching index at $c'$ is $3$. 

Note that the original group $\Gamma=\langle a,b\mid aba^{-1}b^{-1}\rangle$ is a surface group of genus $1$, and $X$ is homeomorphic to a torus. Recall that a branched cover of an oriented surface is again an oriented surface, and we can determine which surface by considering the Euler characteristic of the branched covering. Looking again at Figure~\ref{fig_ex2}, we see that $X(\sigma)$ has a CW-structure with three vertices, six edges, and a single $2$-cell, thus giving us $\chi(X(\sigma))=-2$. Since $\chi(X(\sigma))=2-2g$ where $g$ is the genus, we conclude that $X(\sigma)$ will be a surface of genus $2$.

This can also be seen directly from the fundamental group $\pi_1(X(\sigma))$. Choose $a_3,b_1$ as a spanning tree of the $1$-skeleton, then the generators of $\pi_1(X(\sigma))$ will be $a_1,a_2,b_2$ and $b_3$. After collapsing the spanning tree $a_3\cup b_1$, the attaching map gives rise to the single relator $a_1b_2b_3^{-1}a_1^{-1}b_2^{-1}a_2b_3a_2^{-1}$. This gives us the presentation 
    \begin{align}
        \pi_1(X(\sigma))\cong \langle a_1,a_2,b_2,b_3 \mid a_1b_2b_3^{-1}a_1^{-1}b_2^{-1}a_2b_3a_2^{-1}\rangle
    \end{align}
Let $\alpha_1=a_1$, $\beta_1=b_2$, $\alpha_2=a_2^{-1}b_2a_1$, and $\beta_2=b_3$. By applying Tietze transformations, we obtain the presentation
    \begin{align}
        \pi_1(X(\sigma)) &\cong \langle a_1,a_2,b_2,b_3 \mid a_1b_2b_3^{-1}a_1^{-1}b_2^{-1}a_2b_3a_2^{-1}\rangle\\
        &\cong \langle \alpha_1,\beta_1,\alpha_2,\beta_2 \mid [\alpha_1,\beta_1][\alpha_2,\beta_2]\rangle
    \end{align}
which is the standard presentation of the surface group of genus $2$.
\end{example}

Note that in the previous two examples, the space $X$ we started with was an acceptable $2$-complex. It is informative to consider an example where $X$ is {\bf not} acceptable.

\begin{example}\label{example:projective-plane}
    Consider the finite group  $\mathbb Z_2\cong \langle x \mid x^2\rangle$. The presentation two complex $X$ has a single loop labelled $x$, and a single 2-cell attached by the degree two map $r$ on the circle. $X$ is homeomorphic to the projective plane $\mathbb R P^2$, a closed non-orientable surface, so any branched cover of $X$ would also have to be a closed surface. Now consider the space $X(\sigma)$, where $\sigma = (12\ldots n)$ is the cyclic permutation. The $1$-skeleton of $X(\sigma)$ is then $C_n$, a cycle of length $n$. There is only one lift $\tilde r$ of the attaching map to the $1$-skeleton $C_n$. Observe that the lifted map $\tilde r$ to $C_n$ is either a degree one map if $n$ is even, or a degree two map if $n$ is odd. So the space $X(\sigma)$ is either homeomorphic to $\mathbb RP^2$ (if $n$ odd) or to $\mathbb D^2$ (if $n$ even). Of these $X(\sigma)$, only the $n$ odd case produces a branched cover. In the $n$ even case, the natural map $X(\sigma) \rightarrow X$ has the property that each point has exactly $n$-preimages. The only points where it fails to be a branched cover are along the edges of $X$, where the local topology fails to be preserved. Note, however, that this issue can easily be fixed -- one only needs to change the attaching map so that we have the correct local multiplicity along the edges. This can be done by gluing
    multiple copies of the lifted disk. In the case where $n=2k$ is even, attaching {\it two} copies of a disk along the same attaching map resolves the issue. This gives a space which is topologically an $S^2$, with the two disks identified with the two hemispheres attached to the equator, a combinatorial loop of length $2k$. In terms of the covering of the equator circle, the generator of the deck group $\mathbb Z _{2k}$ acts by the cyclic permutation on the labels, so by a rotation on the equatorial loop. The generator also swaps the two attached disks, so that the centers of the two disks (north and south pole) are now a pair of branch points of index $k$, interchanged by the generator. This realizes $S^2$ as a branched cover of degree $n=2k$ over $\mathbb R P^2$, with a single branch point in $\mathbb R P^2$ having two preimages, each with ramification $k$. We will call this a branched cover {\it with multiplicity} and give more details in Remark~\ref{rem:withmult}.
    
      Building a branched cover with multiplicity is not the only option. Indeed, one can also change the power on an attaching map. In the case where $n=2k$ is even for instance, one can instead attach a single disk using the {\it square} of the lifted map.
    
\end{example}

This last example shows the need for the following:

\begin{lemma}
    Let $X$ be an arbitrary presentation $2$-complex. If no relator is a proper power, then $X(\sigma)\rightarrow X$ is a branched cover.
\end{lemma}

\begin{proof}
     We first note that there is a natural map $X(\sigma) \rightarrow X$ induced by the covering map on the $1$-skeleton, extended to a branched cover from each $2$-cell in $X(\sigma)$ to its image $2$-cell in $X$. It is straightforward to check that every point in $X$ that is {\bf not} the center of a $2$-cell has exactly $n$ pre-images, where $n$ is degree of the covering on the $1$-skeleton. By construction, in the interior of the $2$-cells the map is a branched covering. So we are left with checking the local topology along pre-images of edges, and at vertices.

    In a polygonal $2$-complex, the local topology at a point inside an edge is easy to describe. Let $k$ denote the number of occurrences of that edge (and its inverse) along the boundary labels of all the attached $2$-disks. Then locally, a closed neighborhood of the point is homeomorphic to the product $I \times C_k$, where $C_k$ is the cone over a discrete set of $k$ points. Under this identification the edge corresponds to the product of $I$ times the cone point. 

    So to decide whether the canonical map $X(\sigma) \rightarrow X$ is a covering map along the interior of edges, it suffices to compare, for an edge $e \subset X^{(1)}$ and a fixed pre-image edge $\bar e \subset X(\sigma)^{(1)}$, the number of occurrences of those edges along labels on the boundaries of the disks in $X$ and $X(\sigma)$ respectively. Let $k$ denote the number of occurrences of the edge $e$, and $\bar k$ denote the number of occurrences of the edge $\bar e$. Our goal is to show $\bar k=k$.

    Given a $2$-cell in $X$, obtained by attaching a disk $D$, we can fix an occurrence of the edge $e$ along $\partial D$. Consider all the disks $D_1, \ldots , D_m$ in $X(\sigma)$ lying above $D$. These disks are attached along all the possible lifts of the curve $\partial D$. Corresponding to these lifts, we have covering maps $\partial D_j \rightarrow \partial D$, and the fixed occurrence of the edge $e$ in $\partial D$ has pre-images in the $\partial D_1, \ldots , \partial D_m$. Since these are all possible lifts of $\partial D$, these pre-images of the fixed occurrence of $e$ in $D$ all have distinct indices -- so at most one of these lifts can be $\bar e$. Since each occurrence of $e$ along a disk $D$ gives rise to at most one occurrence of $\bar e$ along one of the lifts of $D$, we obtain the inequality $\bar k \leq k$.
    
    As a cautionary note, observe that without the hypothesis that the complex is acceptable, it is possible that {\it none} of the pre-images has a given label $e_i$ (see Example \ref{example:projective-plane}). So for non-acceptable complexes, the inequality $\bar k\leq k$ could be strict. This is essentially due to the fact that $\sum \deg(\partial D_i\rightarrow \partial D)$ could potentially be smaller than $n$, so that the number of pre-images of $e$ in the $\coprod \partial D_i$ might be less than $n$. This also tells us that, to prove equality, we will need to make use of the ``acceptable'' hypothesis.

    Conversely, for each distinct occurrence of $e$ along a boundary label of a disk in $X$, we can look at the attaching map for that disk, viewed as a loop starting at $e$. We can then lift that loop starting at the chosen pre-image edge $\bar e$, and find the attaching map for a disk in $X(\sigma)$. For a general polygonal $2$-complex, there is no guarantee that the lifted attaching maps starting at distinct occurrences of $e$ give you distinct attaching loops in $X(\sigma)$, see e.g. Example \ref{example:projective-plane}. However, any two distinct occurrences of the edge $e$ in the label of disks $D_i, D_j$ form a sub-overlap. In an acceptable $2$-complex, Lemma \ref{lem-bound-on-overlap-length} tells us this sub-overlap extends to a maximal sub-overlap of bounded length. In particular, the lifted loops based at the two occurrences must eventually diverge, so they correspond to {\bf distinct} lifted disks in $X(\sigma)$. This implies that, in an acceptable $2$-complex, we also have the inequality $\bar k\geq k$, which establishes the equality $\bar k=k$.
    
    Finally, from the covering condition on edges, it is straightforward to check the covering condition also holds at vertices. We leave the details to the interested reader.    
\end{proof}

\begin{remark}\label{rem:withmult}
We conclude this section by discussing how to extend our random model to the case where the original $2$-complex $X$ is {\bf not} acceptable. Consider the situation where a disk is attached along a map that is a proper power in $\pi_1(X^{(1)})$, e.g. the attaching map is along $\alpha = w^k$ for some word $w$ and $k\geq 2$. We may assume that $w$ itself is non-trivial, and not a proper power. As was discussed earlier, the key problem that can arise are issues with the multiplicity of the lifted disk along edges. To analyze this, consider a lift $\tilde w$ of the loop $w$ of degree $d$. Then the corresponding lift $\tilde \alpha$ of the attaching map is the power $\tilde w^r$, where $r=LCM(k,d)$. As a covering of $\alpha$, this lifted map has degree $r/k$ -- which can potentially be different from the desired degree $d$. As discussed in Example \ref{example:projective-plane}, there are several ways to resolve the local multiplicity issue, but we will focus on just one: attaching multiple copies of the lifted disk, each along the map $\tilde \alpha$. In order to get degree $d$, we need to attach $d/(r/k) = kd/r = dk/LCM(k,d)=GCD(k,d)$ copies of the disk along the same map $\tilde \alpha$. We call this the random model {\it with multiplicity}. 
\end{remark}

\medskip

We now have, for each natural number $n$, models that randomly produces a degree $n$ branched cover $X(\sigma)$ of the finite presentation $2$-complex $X$. We will be interested in understanding topological and geometric properties of $X(\sigma)$, as $n$ gets large.

\begin{definition}
Given an event $E=E_n$ depending on a parameter $n$, $E$ holds \textit{asymptotically almost surely} if it holds with probability $1-o(1)$. Thus the probability of success goes to $1$ in the limit as $n\rightarrow\infty$.
\end{definition}

We have defined a random model for acceptable {\it presentation} $2$-complexes. This model can easily be extended to general acceptable $2$-complexes, as detailed in Section \ref{sec:multiple-vertex-case} below.

\vskip 10pt

\subsection{Connectedness of branched covers}

Our random model associates to $t$-tuples of elements in the symmetric group, chosen independently with uniform distribution, a corresponding branched cover. Our goal is now to translate interesting properties of the branched cover $X(\sigma)$ into properties of the $t$-tuple in the symmetric group. We can then hope to leverage our understanding of random elements in symmetric groups to analyze whether or not the property holds for random branched covers in our model. As an easy example, let us consider the connectedness of the branched cover.

\begin{lemma}
    The branched cover $X(\sigma)$ is connected if and only if the subgroup generated by the permutations $\sigma_1 , \ldots ,\sigma_t$ acts transitively on the vertex set.
\end{lemma}

\begin{proof}
    Connectedness of $X(\sigma)$ is completely determined by connectedness of its $1$-skeleton. Given an edge path joining a pair of vertices, one can read off the corresponding product of permutations (and their inverses) taking the initial vertex to the terminal vertex. So if $X(\sigma)$ is connected, then $\langle \sigma_1, \ldots, \sigma_t\rangle$ acts transitively on the vertex set. Conversely, if the subgroup acts transitively on the vertex set, then given any two vertices in $X(\sigma)$ we can find a product of permutations (and their inverses) taking one of these vertices to the other. This then gives us a sequence of edges connecting the two vertices, showing the $1$-skeleton of $X(\sigma)$ is connected. 
\end{proof}

So understanding connectedness of our random branched covers is completely equivalent to understanding when a randomly selected $t$-tuple of elements in $\Sym(n)$ generates a transitive subgroup. This is a classically studied problem, and we have the following result of Dixon \cite{dixon1969probability}:

\begin{proposition}\label{prop.conn}
The proportion of ordered pairs $(\sigma_a,\sigma_b)$, where  $\sigma_a,\sigma_b\in \Sym(n)$ generate a transitive subgroup of $\Sym(n)$ is $1-\frac{1}{n}+O(\frac{1}{n^2})$ as $n\rightarrow\infty$.
\end{proposition}

Translating this result back to our model gives us:

\begin{cor}\label{cor.conn}
    Let $\Gamma=\langle u_1,\dots,u_t \mid r_1,\dots,r_s\rangle$, with corresponding presentation $2$-complex $X$. If $t\geq 2$, then $X(\sigma)$ is asymptotically almost surely connected.
\end{cor}

\begin{proof}
The $t=2$ case follows immediately from Proposition \ref{prop.conn}. For $t\geq 3$, it follows easily from the two generator case, because adding an additional generator just adds more edges to an already connected graph. Equivalently, if the first two elements $\sigma_1, \sigma_2$ already generate a transitive group, then adding additional generators $\sigma_3, \ldots ,\sigma_t$ does not change transitivity of the action.
\end{proof}

\begin{remark}\label{rmk-one-generator}
For the connectedness in Corollary \ref{cor.conn}, we had to assume that there is more than one generator. For the single generator case, random coverings are {\bf not} asymptotically almost surely connected. In particular, the cover will be connected if and only if the chosen permutation is an $n$ cycle. Thus, the probability of connectedness is $\frac{(n-1)!}{n!}$, which goes to $0$ when $n\rightarrow\infty$. 

This is the reason for our requirement that the rank of the $1$-skeleton is $\geq 2$ in our {\bf Main Theorem} (see definition of acceptable $2$-complex). Nevertheless, it is obvious that for a single generator case, any random branched covering has a Gromov hyperbolic fundamental group. Indeed, the presentation $2$-complex has $1$-skeleton consisting of a single loop. Thus any covering of the $1$-skeleton is just a disjoint union of cycles. It follows that each connected component will have fundamental group that is generated by a single element, hence cyclic. It will be isomorphic to $\mathbb{Z}$ if there are no relators, and isomorphic to a (potentially larger) finite cyclic group if there is at least one relator. In either case, the fundamental group is an (elementary) Gromov hyperbolic group.
\end{remark}

\vskip 10pt

\subsection{Overlaps in branched covering spaces}

Let $X$ be the presentation $2$-complex for the presentation 
\[\Gamma=\langle u_1,\dots,u_t \mid r_1,\dots,r_s\rangle\]
where $t\geq 2$
and let $f:X(\sigma)\rightarrow X$ be the $n$-fold branched covering obtained from the $t$-tuple of permutations $\sigma=(\sigma_1,\dots,\sigma_t)$, where $\sigma_i\in \Sym(n),i=1\dots t$. Recall that each disk $D_i\subset X$ corresponding to the relator $r_i$ has finitely many lifts $D_{ij}\subset X(\sigma)$ in the branched cover. For each lift $D_{ij}\subset X(\sigma)$, the index of $D_{ij}$, given by $\ind(D_{ij})$, is the branching index of $f$ at the center of $D_{ij}$.
Clearly, the sum $\sum_{j=1} \ind(D_{ij})=n$ for each $i=1,\dots,s$.

We are interested in studying the small cancellation properties for a random branched cover. We can view a piece $b$ in $X$ (respectively $X(\sigma)$) as a combinatorial path in the $1$-skeleton $X^{(1)}$ (resp. $X(\sigma)^{(1)}$). It is tempting to use the covering map $\rho:X(\sigma)^{(1)} \rightarrow X^{(1)}$ to compare pieces in $X(\sigma)$ with pieces in $X$. Unfortunately, the image of piece in $X(\sigma)$ might not be a piece in $X$, as it might map to a path that has length greater than the boundary of the disks. Similarly, the connected lift of a piece in $X$ might not be a piece in $X(\sigma)$, as it might be properly contained in an overlap in $X$. For this reason, it is more convenient to work with overlaps. As overlaps are defined in terms of the universal cover of the attaching loops, these behave better with respect to the branched covering map $\rho$ restricted to the one skeleton. Indeed, we have the following:

\begin{lemma}\label{lem-lifting-overlaps}
    Let $X$ be an acceptable presentation $2$-complex,
    $X(\sigma)$ a branched cover of $X$, and $\rho: X(\sigma) \rightarrow X$ the branched covering map. Then:
    \begin{itemize}
        \item if $(\bf{p}, \bf{p'})$ is an overlap for the pair of $2$-cells $D, D'$ in $X$, then $(\bf{p}, \bf{p'})$ is an overlap for $X(\sigma)$ for a pair of disks $\hat D, \hat D'$ that are branched covers of $D, D'$ respectively;
        \item if $(\bf{\hat p}, \bf{\hat p'})$ is an overlap in $X(\sigma)$ for the pair of disks $\hat D, \hat D'$, then the pair defines an overlap for the image pair of disks $D=\rho(\hat D), D'=\rho(\hat D')$ in $X$.
    \end{itemize}
\end{lemma}

\begin{proof}
Recall that the overlap $(\bf{p}, \bf{p'})$ is actually a pair of (equivalence classes of) subpaths ${\bf p} \subset \widetilde{\partial D}$ and ${\bf p'} \subset \widetilde{\partial D'}$, with the property that the lifted attaching maps $\tilde \alpha:\widetilde{\partial D} \rightarrow X^{(1)}$, $\tilde \beta :\widetilde{\partial D'} \rightarrow X^{(1)} $ coincide on the subpaths $\bf{p}, \bf{p'}$. Given any pre-image $v_i$ of the unique vertex $v\in X^{(1)}$, covering space theory tells us we can lift the maps $\tilde \alpha, \tilde \beta$ to maps $\bar \alpha : \widetilde{\partial D} \rightarrow X(\sigma)^{(1)}, \bar \beta: \widetilde{\partial D'} \rightarrow X(\sigma)^{(1)}$ based at the vertex $v_i$. Since the original maps $\tilde \alpha, \tilde \beta$ coincide on the subpaths $\bf{p}, \bf{p'}$, the lifted maps will have the same property. Moreover, since the original maps $\tilde \alpha, \tilde \beta$ {\it differ} on the two edges immediately preceding (respectively following) the subpaths $\bf{p}, \bf{p'}$, the same property holds for the lifted maps. Finally, we note that the lifted maps $\bar \alpha, \bar \beta$ are periodic, as they will cover one of the connected lifts $\hat \alpha, \hat \beta$ of the attaching maps $\alpha, \beta$ (the lifts based at the vertex $v_i$). These lifted attaching maps 
$\hat \alpha: \widehat{\partial D} \rightarrow X(\sigma)^{(1)}$, $\hat \beta: \widehat{\partial D'} \rightarrow X(\sigma)^{(1)}$ 
are defined on finite covers $\widehat{\partial D} \rightarrow \partial D$ and $\widehat{\partial D'} \rightarrow \partial D'$. These 
define a pair of $2$-cells $\hat D, \hat D'$ in $X(\sigma)$. We conclude that the pair $(\bf{p}, \bf{p'})$ is an overlap for the $2$-cells $\hat D, \hat D'$.

Conversely, if we have an overlap $(\bf{p}, \bf{p'})$ in $X(\sigma)$ for a pair of $2$-cells $\hat D, \hat D'$, we can project the overlap via the covering map $\rho$. More precisely, from the construction of $X(\sigma)$, the attaching maps $\hat \alpha: \partial \hat D \rightarrow X(\sigma)^{(1)}$, $\hat \beta: \partial \hat D' \rightarrow X(\sigma)^{(1)}$ are lifts of the attaching maps $\alpha: \partial D \rightarrow X^{(1)}$, $\beta: \partial D' \rightarrow X^{(1)}$ for the pair of $2$-cells $D, D'$ in $X$. This means there are finite covering maps $\pi: \partial \hat D \rightarrow \partial D$, $\pi': \partial \hat D' \rightarrow \partial D'$, and commutative diagrams $\alpha \circ \pi = \rho \circ \hat \alpha$, $\beta \circ \pi' = \rho \circ \hat \beta$. 

The covering map $\pi$ allow us to identify the universal covers of $\partial \hat D$ and $\partial D$, via the lift of the covering map $\widetilde \pi: \widetilde{\partial \hat D} \rightarrow \widetilde{\partial D}$ (and similarly for $\pi'$, $\partial \hat D'$, and $\partial D'$). With this identification, the map $\bar \alpha: \widetilde{\partial \hat D} \rightarrow X(\sigma)^{(1)}$ descends to a map $\tilde \alpha:=\rho \circ \bar \alpha \circ \tilde \pi^{-1}: 
\widetilde{\partial D} \rightarrow X^{(1)}$. Similarly, we have a map $\tilde \beta := \rho \circ \bar \beta \circ (\tilde \pi')^{-1}: \widetilde{\partial D'} \rightarrow X^{(1)}$. Since the maps $\bar \alpha, \bar \beta$ agree on the subpaths $\bf{p}, \bf{p'}$ but differ on the immediately preceding (and immediately following) edges, the same property is true for the composite maps $\rho \circ \bar \alpha, \rho \circ \bar \beta$. Using the identification $\tilde \pi, \tilde \pi'$ of universal covers, we can view $\bf{p}, \bf{p'}$ as subpaths in $\widetilde{\partial D}, \widetilde{\partial D'}$. This shows that $(\bf{p}, \bf{p'})$ defines an overlap for the $2$-cells $D, D'$, completing the proof of the Lemma.
\end{proof}

An immediate consequence of the lemma is the following

\begin{cor}\label{rmk-piecelength}  
    Let $X$ be an acceptable presentation
    $2$-complex, and $X(\sigma)$ a branched cover of $X$. If the $2$-cell $\bar D$ in $X(\sigma)$ is an index $k$ branched cover of the $2$-cell $D$ in $X$, then the overlap ratios are related by $o(\bar D) = o(D)/k$. In particular, $o(X(\sigma)) \leq o(X)$.
\end{cor} 

\begin{proof}
    From the lemma, we see that the lengths of overlaps for $D$ coincide with the lengths of overlaps for $\bar D$. Since $\ell(\partial \bar D) = k\cdot \ell(\partial D)$, the result follows. 
\end{proof}

As we remarked earlier, the small cancellation condition $C'(\lambda)$ (where $0<\lambda <1$) for a $2$-cell $D$ is implied by $o(D)<\lambda$. So we immediately obtain:

\begin{cor} \label{lemma:sc}
If a $2$-cell in $X$ satisfies the $o(D)<\lambda$, then all of its lifts in $X(\sigma)$ satisfy $C'(\lambda)$. 
\end{cor}

\begin{remark}\label{good-worrisome-remark}
Consider a $2$-cell $D\subset X$ that {\bf does not} satisfy  $o(D)< \lambda$. If a lift $\overline{D}\subset X(\sigma)$ has index $k$ satisfying $k> \frac{o(D)}{\lambda}$, then by Corollary \ref{rmk-piecelength} we have
    $$o(\bar D) = \frac{o(D)}{k} < \lambda$$ 
so $\overline{D}$ satisfies $C'(\lambda)$ in $X(\sigma)$.
\end{remark}

We denote by $R_L$ (respectively $R_S$) the length of the longest (resp. shortest) relation in the presentation $\Gamma$. Let us introduce some constants associated to the presentation $2$-complex $X$.
    
By Lemma~\ref{lem-bound-on-overlap-length} there is a uniform bound on the length of overlaps in $X$. We introduce the parameter $\mathcal O := R_L^2 +R_L$, which serves as a global upper bound on the length of overlaps in $X$. In view of Lemma~\ref{lem-lifting-overlaps}, $\mathcal O$ also serves as an upper bound on the length of overlaps in {\bf any} of the branched covers $X(\sigma)$. Since all the relators in $\Gamma$ have length $\geq R_S$, we also obtain the upper bound on the overlap ratio $o(X) < \frac{\mathcal O}{R_S}$. We are looking for branched covers with overlap ratio less than $\lambda$. To this end, let us introduce the critical index $I:=\frac{\mathcal O}{\lambda R_S}$. From Remark \ref{good-worrisome-remark}, we know that any disk $D$ in a branched cover $X(\sigma)$ whose index is $\geq I$ automatically has $o(D)<\lambda$, thus satisfies $C'(\lambda)$. 

\begin{definition}\label{disks}
Given an acceptable presentation
$2$-complex $X$, and $\lambda\in(0,1)$, a disk $D$ in a branched cover $X(\sigma)$ is called a \textit{$\lambda$-good disk} if its index is greater than or equal to $I$. A disk that is not a good-disk is called a \textit{$\lambda$-worrisome disk}. We will typically be working with a fixed value of $\lambda$, and refer to $\lambda$-good disks as {\it good disk} and $\lambda$-worrisome disk as a {\it worrisome disk}.
\end{definition}

All the $\lambda$-good disks satisfy $C'(\lambda)$ small cancellation in $X(\sigma)$. However, this is not true for $\lambda$-worrisome disks, which may or may not satisfy the $C'(\lambda)$ small cancellation condition.

\vskip 10pt

\subsection{Disks in random branched covering spaces}

For the relators $r_1,\dots,r_s$ and a random choice of permutations $\sigma=(\sigma_1,\dots,\sigma_t)$, $\sigma_i\in \Sym(n)$, we define another type of permutation $r_i(\sigma)\in \Sym(n)$ that represents the structure of lifts of the disk $D_i$. In $X(\sigma)$, we first attach a disk for a lift of $D_i$ of the relator $r_i$ starting from the vertex $v_1$. Let $v_{i(1)}:=v_1$ and let $v_{i^j(1)}$ be the vertex that is arrived at after following the letters of $r_i$ a total of $j$ times. Let $k\geq 1$ be the smallest such that $v_{i^k(1)}=v_1$. Then $(v_{i(1)} \cdots v_{i^k(1)})$ forms a cycle of length $k$ with entries corresponding to the indices of the vertices appearing in the procedure. We repeat the same process at a vertex that is not already obtained in a previous cycle until there are no vertices remaining. Finally we get a permutation of $n$ elements and denote it $r_i(\sigma)$. One checks the cycle lengths of $r_i(\sigma)$ have a one-to-one correspondence with the indices of lifts of $D_i$. The permutation $r_i(\sigma)$ is the result of applying the {\it word map} $r_i : \Sym(n) \times \cdots \times \Sym(n) \rightarrow \Sym(n)$ (see \cite{hanany2023word}) to the $t$-tuple $\sigma$ of permutations. We will use the following result from \cite[Corollary 1.5]{hanany2023word}:

\begin{proposition}\label{expnumber}
    Let $k\geq 2$ be a fixed integer. If the permutation $r_i$ is not a proper power, then the expected number of cycles of length $k$ in $r_i(\sigma)$ is $\frac{1}{k}+\mathbf{O}(n^{-\pi(r_i)})$, where $\pi(r_i)\geq 2$ is the primitive rank defined in \cite{hanany2023word}. 
\end{proposition}

\medskip

Let $L_n(k)$ be the total number of cycles of length at most $k$ in all of the permutations $r_1(\sigma),\dots,r_s(\sigma)\in Sym(n)$. Note that, by hypothesis, none of our $r_i$ are proper powers, so applying Proposition~\ref{expnumber}, we have
 \begin{align*}
     \mathbb{E}\big(L_n(k)\big)=s(1+\frac{1}{2}+\dots+\frac{1}{k})+\mathbf{O}(n^{-\pi(r_1)}+\dots+n^{-\pi(r_s)}).
 \end{align*}
 As a result, when $n\rightarrow\infty$ the expected value $\mathbb{E}(L_n(k))\rightarrow s(1+\frac{1}{2}+\dots+\frac{1}{k})$. Knowing the asymptotics of the expected value allows us to deduce information about the tails of the probability distributions.

\begin{lemma}\label{p_n,m}
    For any $\epsilon>0$, there exists $N,m\in\mathbb{N}$ such that if $n\geq N$, then 
$$\mathbb{P}\big(L_n(k)\leq m\big)>1-\frac{\epsilon}{2}.$$
\end{lemma}

\begin{proof}
Suppose not. Then there exists $\epsilon>0$ with the property that for each $m$, we can find a sequence $n_j\rightarrow \infty$ such that $\mathbb{P}\big(L_{n_j}(k)\geq  m+1\big)>\frac{\epsilon}{2}$. We can now estimate the expectation of $L_{n_j}(k)$ from below:
\begin{align*}
     \mathbb{E}\big(L_{n_j}(k)\big)&=\sum_{i=0}^\infty i\mathbb{P}\big(L_{n_j}(k)=i\big)\\ &\geq\sum_{i=m+1}^\infty (m+1) \cdot \mathbb{P}\big(L_{n_j}(k)=i\big)\\ 
     &=(m+1) \cdot \mathbb{P}\big(L_{n_j}(k)\geq m+1\big)>(m+1) \cdot \frac{\epsilon}{2}
 \end{align*}
For a fixed choice of $m$, this estimate holds for all the $n_j$ in the sequence. Applying this to the specific case where $m=\frac{2s}{\epsilon}\cdot \big(1+\frac{1}{2}+\dots+\frac{1}{k}\big)$, we obtain an infinite sequence of integers $n_j\rightarrow \infty$ where
\begin{align*}
     \mathbb{E}\big(L_{n_j}(k)\big)&>(m+1) \cdot \frac{\epsilon}{2}\\
     &> s\big(1+\frac{1}{2}+\dots+\frac{1}{k}\big) + \frac{\epsilon}{2}
 \end{align*}

On the other hand, we saw earler that the expected value $\mathbb{E}(L_{n}(k))$ converges to $s(1+\frac{1}{2}+\dots+\frac{1}{k})$ as $n\rightarrow\infty$, giving us a contradiction. 
\end{proof}

Note that in the proof above, we are not assuming the existence of limiting distribution of $L_n(k)$, but using the expected value of $L_n(k)$ for finite $n\in\mathbb{N}$ and its limit as $n\rightarrow \infty$. We can interpret Proposition~\ref{expnumber} and Lemma \ref{p_n,m} in terms of random branched coverings. 

\begin{cor}\label{cor-number-of-small-disks}
    Let $X$ be an acceptable presentation $2$-complex. Then given any integer $k$, and $\epsilon >0$, we can find an integer $M = M(k, \epsilon)$ with the following property: for any $n$ sufficiently large, with probability at least $1-\frac{\epsilon}{2}$ a random degree $n$ branched cover $X(\sigma)$ contains at most $M$ disks of index less than or equal to $k$.
\end{cor}

Next we turn our attention to topological properties of disks 

\begin{lemma}\label{emb}
Let $X$ be an acceptable presentation  $2$-complex, and $m$ a given integer. Then asymptotically almost surely the random branched covers of $X$ have all disks of index $m$ that are injectively embedded.
\end{lemma}

\begin{proof}
We will count all possible random branched coverings and see how many of them contain non-injective lifts of index $m$. Since the number of generators is $t$ and $|\Sym(n)|=n!$, the number of choices for permutations $\sigma=(\sigma_1,\dots,\sigma_t)$ is $(n!)^{t}$. In other words, there are $(n!)^{t}$ random branched coverings of $X$. 

Let $r$ be a relator. By the Proposition~\ref{expnumber}, the expected number of lifts of $r$ of index $m$ in a random branched covering is $\frac{1}{m}+\mathbf{O}(n^{-\pi(r)})$. Since there are $(n!)^{t}$ random branched coverings, the total number of lifts of $r$ of index $m$ in all possible random branched coverings is $(n!)^{t}\left(\frac{1}{m}+\mathbf{O}(n^{-\pi(r)})\right)$.

Now, in all possible branched coverings, we count the total number of injective lifts of the relator $r$ of index $m$. Let $r=w_1\dots w_{|r|}$ where each \[w_i\in\{u_1,\dots,u_t,u_1^{-1},\dots,u_t^{-1}\},\] the symmetric generating set. For such a lift $D\subset X(\sigma)$ of $r$, since it has index $m$, the length of $\partial D$ is $m|r|$. 

To be injective, $\partial D$ has to contain $m|r|$ distinct vertices. Assume that $n$ is sufficiently large so that $n \geq m|r|$. Choosing the $m|r|$ distinct vertices in order amounts to $n(n-1)\dots(n-m|r|+1)$ possibilities. Once we choose the labels on the vertices of $\partial D$, the labels on the oriented edges along $\partial D$ will be determined by $r=w_1\dots w_{|r|}$ and our labeling convention (see Remark \ref{rmk:label-convention}). Thus there are $n(n-1)\dots(n-m|r|+1)$ different ways of labeling the boundary of an injective lift $D$. Again, we remark that such a labeled lift may occur in different branched covers. Note that since cyclic permutations of the labeling on disk's boundary represent the same lift, there are $\frac{1}{m}n(n-1)\dots(n-m|r|+1)$ different ways of labeling the boundary up to cyclic relabeling.

Now we count all possible branched coverings $X(\sigma)$ that contain the choice of $\partial D$ with a fixed given injective labeling. Let $\ell_i$ be the number of occurrences of a lift of the generator $u_i$ along $\partial D$. Obviously, $\sum_{i=1}^t\ell_i=\ell(\partial D)=m|r|$. For the remaining $n-l_i$ lifts of $u_{i}$ that are not contained in the labeled $\partial D$, there are $n-\ell_i$ possible initial vertices for the lifts, and $n-\ell_i$ possible ending vertices for the lifts. Hence the number of ways to place the remaining lifts of $u_i$ boils down to choosing a pairing between the possible initial vertices and terminal vertices. There are $(n-\ell_i)!$ such pairings. Ranging over all the edges, we obtain $(n-\ell_1)!\dots(n-\ell_t)!$ labeled branched coverings that contain the lift $D$ with the prescribed (injective) labeling on $\partial D$.

Thus in all possible branched coverings, the total number of injective lifts of $r$ of index $m$ will be
\begin{align*}
    \frac{\big[n(n-1)\dots(n-m|r|+1)\big]\cdot \big[(n-\ell_1)!\dots(n-\ell_t)!\big]}{m}
\end{align*}
and therefore the total number of non-injective lifts of $r$ in all possible branched coverings is

\begin{align*}
    \frac{(n!)^{t}}{m}-\frac{\big[n(n-1)\dots(n-m|r|+1)\big]\cdot \big[(n-\ell_1)!\dots(n-\ell_t)!\big]}{m}+(n!)^t\mathbf{O}(n^{-\pi(r)}).
\end{align*}

The total number of branched coverings that contain a non-injective lifts will be bounded above by the total number of non-injective lifts in all branched coverings. Let us denote by $P_n(r)$ by the probability that a branched covering contains a non-injective lift of $r$ of index $m$. Then using the above estimate we obtain the upper bound
\begin{align*}
    &P_n(r)\leq\frac{(n!)^{t}-n(n-1)\dots(n-m|r|+1)(n-\ell_1)!\dots(n-\ell_t)!}{m(n!)^t}+\mathbf{O}(n^{-\pi(r)}).
\end{align*}
With this upper bound in hand, we can easily compute the limit of $P_n(r)$ as $n\rightarrow\infty$: 
\begin{align*}
    \lim_{n \to \infty}P_n(r) &\leq \lim_{n \to \infty}\frac{(n!)^{t}-n(n-1)\dots(n-m|r|+1)(n-\ell_1)!\dots(n-\ell_t)!}{m(n!)^t}\\
    & =\frac{1}{m}- \frac{1}{m}\cdot\lim_{n \to \infty}\frac{n(n-1)\dots(n-m|r|+1)}{\left(\displaystyle\prod_{i=0}^{\ell_1-1}(n-i)\right)\dots\left(\displaystyle\prod_{i=0}^{\ell_t-1}(n-i)\right)}.\\
\end{align*}
Note that in the last term, the numerator and the denominator are both  $m|r|$-degree monic polynomials in $n$. As $n\rightarrow \infty$, the ratio tends to one, and as a result $\lim_{n \to \infty}P_n(r) = 0$. This tells us that in a random branched covering, all lifts of the single relator $r$ of index $m$ are embedded asymptotic almost surely. 

The probability that a branched covering contains a non-injective index $m$ lift of one of the finitely many relators $r_1,\dots,r_s$ is less than $\sum_{i=1}^s P(r_i)$. Since the sum goes to zero as $n\rightarrow\infty$, we conclude that every lift of an $r_i$ index $m$ are embedded in a random branched covering asymptotic almost surely.
\end{proof}



A similar argument can be used to control intersections of disks.

\begin{lemma}\label{dd>1}
    Let $X$ be an acceptable presentation $2$-complex, and $I$ a given integer. Then asymptotically almost surely in the random branched covers of $X$ all disks of index at most $I$ are pairwise disjoint. 
\end{lemma}

\begin{proof}
    We want to compute the probability that two lifts of index at most $I$ intersect in some branched cover and show that it goes to $0$ as $n\rightarrow\infty$. Among the $(n!)^t$ possible labeled branched coverings, this amounts to counting the ones that contain a pair $(D, D')$ of lifts of index at most $M$ that intersect non-trivially, and showing that the proportion of such coverings tends to 0 as $n$ tends to infinity. 
    
    The idea of the proof is very similar to the proof of Lemma \ref{emb}. We will provide a very crude upper bound on the number of such covers, proceeding in two steps. The first step is to identify and count the possible images of the boundaries of the intersecting disks inside all possible branched covers. The second step is to look at each such possible image, and estimate the number of covers that contain that specific pattern of intersection.
    
    We now provide the details for the first step. Consider a pair $r, r'$ of relators, and assume that we have disks $D, D'$ that are lifts of the relators $r, r'$ of respective fixed index $m, m' \leq I$. Without loss of generality, assume that $m\leq m'$. We want to identify the possible intersecting images of $\partial D$ and $\partial D'$ inside a branched cover. Note that by Lemma \ref{emb}, any such lift of a disk with index at most $M$ is injective asymptotically almost surely. So we may assume that the images are injective on both $\partial D$ and $\partial D'$.

    By hypothesis, we have that the two images have non-empty intersection. Any such intersection gives rise to an overlap. Suppose that the number of overlaps between $r$ and $r'$ is $\alpha$. By Lemma \ref{lem-lifting-overlaps}, every overlap between $\partial D$ and $\partial D'$ must be lifts of one of the overlaps between $r, r'$. Since $\partial D$ is an $m$-fold cover of $r$, there are at most $m\alpha$ possible paths in $\partial D$ that can be part of an overlap. So in the lift $\partial D\cap\partial D'$ cannot contain more than $m\alpha$ overlaps. Fix $1\leq \beta \leq m\alpha$ and consider the case that $\partial D$ and $\partial D'$ contain exactly $\beta$ overlaps. By abuse of notation, we let $\{(\bf{p_1}, \bf{q_1}), \ldots,(\bf{p_\beta}, \bf{q_\beta})\}$ be the collection of images of the overlaps inside $\partial D$ and $\partial D'$ (recalling that each $\bf{p_i}$ actually lives in $\widetilde{\partial D}$ and similarly for $\bf{q_i}$). Note that, since the map on $\partial D'$ is injective, the subpaths $\{\bf{p_1}, \ldots , \bf{p_\beta}\}\subset \partial D$ must be pairwise disjoint. Similarly, injectivity of the map on $\partial D$ implies the subpaths $\{\bf{q_1}, \ldots , \bf{q_\beta}\}\subset \partial D'$ must be pairwise disjoint. Finally, $Z$ is obtained as a quotient space of $\partial D \coprod \partial D'$ by identifying each of the subpaths $\bf{p_i} \subset \partial D$ with the corresponding $\bf{q_i}\subset \partial D'$, see Figure \ref{fig-construct-Z}.

    \begin{figure}[h!]
    \includegraphics[width=8cm,height=5.33cm]{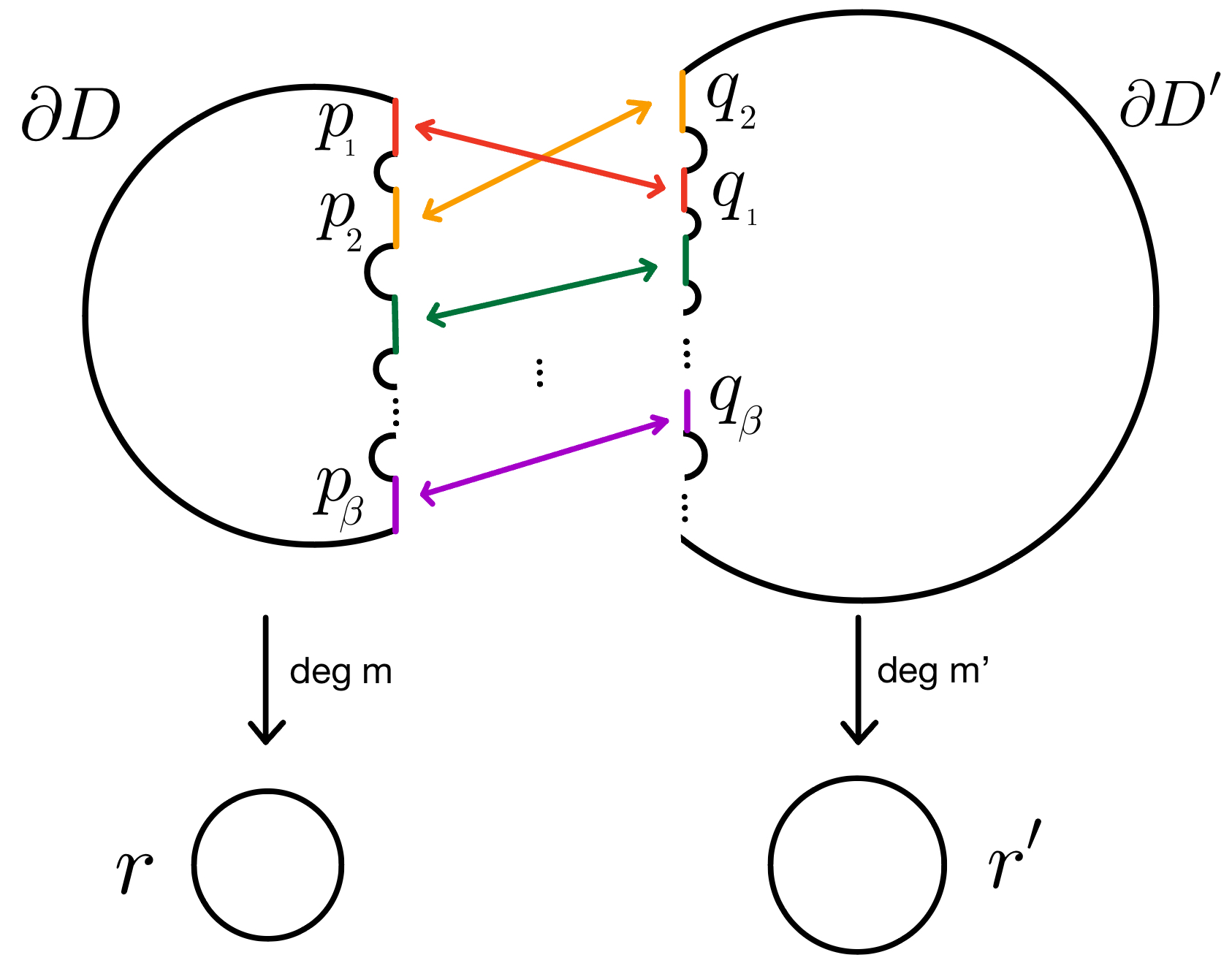}
    \caption{Construction of space $Z$}\label{fig-construct-Z}
    \end{figure}

    Before proceeding to the second step, let us count the number of such $Z$ we can construct. Since we are only interested in obtaining a crude upper bound, the number of such $Z$ is certainly no larger than $$\binom{m\alpha}{\beta}\binom{m'\alpha}{\beta}\cdot \beta!$$
    This estimate comes from the choices for the overlap paths $p_1, \dots ,p_\beta$ for $\partial D$, the choices for the overlap paths $q_1, \dots ,q_\beta$ for $\partial D'$, and the choices of pairings between the $p_i$ and $q_i$. This bounds the number of $Z$ produced, where $\partial D, \partial D'$ are degree $m$, $m'$ covers of $r, r'$, with exactly $\beta$ overlaps.

    If we want to account for {\bf all} possible overlapping disks, we need to also take into account the choices of relators $r, r'$, the choice of degree of covers $m, m'$, and the choice of the number of overlaps $\beta$. let $A$ denote the uniform upper bound on the number of overlaps occurring along any of the relations (see Corollary \ref{cor-bound-on-number-of-overlaps}); this quantity only depends on $X$. Then we have the inequalities $m\alpha \leq IA$, $m'\alpha \leq IA$, and $\beta \leq IA$. Thus the upper bound obtained in the previous paragraph satisfies the following uniform upper bound:
     $$\binom{m\alpha}{\beta}\binom{m'\alpha}{\beta}\cdot \beta! \leq \binom{IA}{\frac{1}{2}IA}^2\cdot (IA)!$$
    Assume our presentation $2$-complex $X$ has exactly $s$ $2$-cells (i.e. the group has $s$ relations). Then there are at most $s^2$ choices for the relations $r, r'$, at most $I^2$ choices for the degrees $m, m'$, and at most $IA$ choices for the number $\beta$ of overlaps. Along with the uniform upper bound given above, we see that the total possible number of spaces $Z$ we can construct is bounded above by:
    \begin{align}\label{eqn-number-of-terms}
        s^2I^3A\binom{IA}{\frac{1}{2}IA}^2 \cdot (IA)!
    \end{align}
    a number that only depends on the acceptable presentation $2$-complex $X$ and the upper bound $I$ on the index of disks. This completes the first step of the argument. 

    \vskip 5pt

    For the second step, we fix one of the $Z$ constructed above, and denote by $P_n(Z)$ the probability that one of the degree $n$ covers of $X^{(1)}$ contains an embedded copy of $Z$. We then have that the probability that some degree $n$ branched cover contains a pair of intersecting disks of index at most $I$ is bounded above by
    $\sum_Z P_n(Z)$. Note that the number of terms in this sum is uniformly bounded above independent of $n$, see Equation (\ref{eqn-number-of-terms}). Thus to complete the proof it suffices to show that, for each fixed $Z$, $\lim_{n\rightarrow \infty} P_n(Z) =0$. 

    \vskip 5pt
    
    To estimate $P_n(Z)$, we count the number of degree $n$ covers of $X^{(1)}$ that contain an injective copy of $Z$. To construct such a cover, we first define the injective map on $Z$. This amounts to assigning a distinct label to each vertex of $Z$. Recall that $Z$ is built from a combinatorial cycle of length $m|r|$ and a combinatorial cycle of length $m'|r'|$, by identifying together $\beta$ pairwise disjoint paths $p_i\subset \partial D$ with $\beta$ pairwise disjoint paths $q_i \subset \partial D'$. Each path of length $|p_i|=|q_i|$ has exactly $|p_i|+1$ vertices, so the total number of vertices in the graph $Z$ is exactly
    $$m|r| + m'|r'| - \big(\beta +\sum |p_i|\big)$$
    The number of ways to label the vertices of $Z$ with distinct labels is thus exactly:
    \begin{align}\label{egn-F-degree}
        F_Z(n):=n \cdot(n-1) \cdot (n-2) \cdot \ldots \cdot \Big(n+1-\big(m|r| + m'|r'| - \beta - \sum |p_i|\big)\Big)
    \end{align}
    Observe that $F_Z$ is a polynomial in $n$.

    Once we have identified the image of $Z$, to construct a labeled branched cover containing $Z$ we have to consider the remaining edges. Let $a_i$ be number of lifts of the generator $u_i$ that appear in the label of $Z$. Obviously, every edge in $Z$ is labeled by the lift of one of the generators, so we have 
    \begin{align}\label{eqn-G-degree}
        \sum_{i=1}^{t}a_i =m|r|+m'|r'|-(\sum_{j=1}^\beta|p_i|).
    \end{align}
    Then there are $(n-a_i)$ lifts of the $u_i$-edge that still remain to be determined. There are $(n-a_i)$ choices for the initial vertex of each lifted edge, and $(n-a_i)$ choices for the terminal vertex. Any pairing of the initial vertices with terminal vertices will give a valid covering of the $u_i$-edge. Thus there are $(n-a_i)!$ possibilities for the full lift of the $u_i$-edge. As we range over all the generators, we see that the number of ways to complete the labeling of $Z$ to a covering of $X^{(1)}$ is 
    $$\prod_{i}(n-a_i)!$$
    This means that the probability $P_n(Z)$ we wanted to compute is given by:
    $$P_n(Z)  =  \frac{F_Z(n) \cdot\prod_{i=1}^t(n-a_i)!}{(n!)^t}.$$
    Defining the polynomial $G_Z(n)$ via
    $$G_Z(n):=\frac{(n!)^t}{\prod_{i=1}^t(n-a_i)!}$$
    we see that $P_n(Z) = F_Z(n)/G_Z(n)$ is a rational function of $n$. The asymptotics of $P_n(Z)$ as $n\rightarrow\infty$ is then completely determined by comparing the degrees of $F_Z(n)$ and $G_Z(n)$. From the definition of $F_Z(n)$ in Equation (\ref{egn-F-degree}), we immediately obtain
    $$\deg(F_Z) = m|r|+m'|r'|-\beta - \sum|p_i|.$$
    On the other hand, the definition of $G_Z$ and Equation (\ref{eqn-G-degree}) tells us that
    \begin{align*}
        \deg(G_Z)=\sum_{i=1}^t a_i =m|r|+m'|r'|-\sum|p_i|
    \end{align*}
    Since $\beta\geq 1$, we see that $\deg(G_Z(n))>\deg(F_Z(n))$. It follows that $\lim_{n\rightarrow \infty} P_n(Z) =0$, completing the proof.

\end{proof}

\begin{cor}\label{d-cor}
    The property that all worrisome-disks are pairwise disjoint in $X(\sigma)$ holds asymptotically almost surely.
\end{cor}

\bigskip

\section{Proof of the Main Theorem}\label{Section:Proof}

Let $\Gamma=\langle u_1,\dots,u_t \mid r_1,\dots,r_s\rangle$ be a finitely presented group and $X$ the presentation $2$-complex where $t\geq 2$. Let $f:X(\sigma)\rightarrow X$ be a $n$-fold random branched covering obtained from random permutations $\sigma=(\sigma_1,\dots,\sigma_t)$ where $\sigma_i\in \Sym(n),i=1\dots t$. 

We are now ready to prove the {\bf Main Theorem} in the special case of the presentation $2$-complex. For the convenience of the reader, we restate the special case. 

\begin{thm}
    Let $X$ be the presentation $2$-complex associated to the finite presentation $\Gamma=\langle u_1,\dots,u_t \mid r_1,\dots,r_s\rangle$. We assume that the relators are cyclically reduced, that none of the relators are conjugate to proper powers, and that no relations are conjugate to each other. Let $X(\sigma)$ be an $n$-fold random branched cover of $X$. Then for any given constant $\lambda>0$, $X(\sigma)$ is asymptotically almost surely homotopy equivalent to a $2$-complex satisfying geometric $C'(\lambda)$-small cancellation.
\end{thm}

\begin{proof}
    Observe that the presentation $2$-complex $X$ is an acceptable $2$-complex. Given an $\epsilon >0$, we want to show that for all sufficiently large $n$, a random branched cover $X(\sigma)$ is homotopy equivalent to a $2$-complex satisfying $C'(\lambda)$ with probability greater than $1-\epsilon$. 
    
    We remind the reader of some constants associated to the presentation $2$-complex $X$ and the given $\lambda >0$:
    \begin{itemize}
        \item $R_L$ and $R_S$ denote the length of the longest (resp. shortest) relator in the presentation $\Gamma$;
        \item the parameter $\mathcal{O}:= R_L^2+R_L$ serves as a uniform upper bound on the length of overlaps in $X$ and in any $X(\sigma)$ (see Lemma \ref{lem-bound-on-overlap-length} and Lemma \ref{lem-lifting-overlaps});
        \item the critical index $I:=\frac{\mathcal O}{\lambda R_S}$ has the property that any disk $D$ whose index is is at least $I$ automatically satisfies $C'(\lambda)$.
    \end{itemize}
    Next we introduce some parameters that also depend on the given $\epsilon$. Recall from Corollary \ref{cor-number-of-small-disks} that we can find an integer $M:=M(I, \epsilon)$ with the property that for $n$ sufficiently large, a random branched cover $X(\sigma)$ will contain at most $M$ disks of index $\leq I$ with probability at least $1-\epsilon/2$. Finally, we set the parameter $K$ to be
    \begin{equation}\label{eqn-K}
        K:=R_S^{-1}(1+\lambda^{-1})\mathcal O\big(M^2I(R_L\mathcal O)^2\big)+R_S^{-1}\lambda^{-1}\mathcal O
    \end{equation}
    Thus the choice of $K$ depends (via $M$) on the choice of $\epsilon$.
    
     From Corollary \ref{cor-number-of-small-disks}, Lemma~\ref{emb}, and Lemma~\ref{dd>1}, we can choose an $N$ sufficiently large, so that for all $n>N$, with probability greater than $1-\epsilon$, a random branched cover $X(\sigma)$ has the following three properties:
    \begin{enumerate}
        \item the number of disks in $X(\sigma)$ of index at most $I$ is at most $M$;
        \item all disks of index at most $K$ are injective;
        \item all disks of index at most $K$ are pairwise disjoint.
    \end{enumerate}
    So to complete the proof, it suffices to show that these branched covers are homotopy equivalent to a complex satisfying $C'(\lambda)$-small cancellation.

    Observe that, if $X(\sigma)$ has {\bf no} disks of index $\leq I$, then all disks in $X(\sigma)$ have overlap ratio less than $\lambda$ (see Remark \ref{good-worrisome-remark}) and $X(\sigma)$ itself satisfies $C'(\lambda)$-small cancellation. But in general, for most $\sigma$ the space $X(\sigma)$ will contain some disks of index $\leq I$. To analyze these branched covers, we partition the disks in $X(\sigma)$ according to their index:
    \begin{itemize}
        \item {\it small disks} are those with index at most $I$,
        \item {\it medium disks} are those with index $>I$ but $\leq K$
        \item {\it large disks} have index greater than $K$.
    \end{itemize}
    By the definition of worrisome disks, the set of small disks are exactly the same as the set of worrisome disks. Since these are the disks which might have overlap ratio greater than $\lambda$, we construct a new space $Y(\sigma)$ by collapsing each of the small disks in $X(\sigma)$ to a point. There is a natural quotient map $q: X(\sigma) \rightarrow Y(\sigma)$. 

    \vskip 5pt
    
    \noindent {\bf Fact 1:} The map $q$ is a homotopy equivalence.

    \vskip 5pt
    
    To check that the quotient map is a homotopy equivalence, recall that quotienting out a contractable subcomplex from a CW-complex yields a homotopy equivalence. From property (2), small disks are embedded, hence have image in $X(\sigma)$ that are homeomorphic to $\mathbb D^2$. From property (3), small disks are pairwise disjoint. Collapsing them one by one yields a finite sequence of homotopy equivalences from $X(\sigma)$ to $Y(\sigma)$.

    Next we need to establish that $Y(\sigma)$ satisfies $C'(\lambda)$-small cancellation. It suffices to check that all the overlap ratios of disks in $Y(\sigma)$ are less than $\lambda$. Since disks in $Y(\sigma)$ are images of disks in $X(\sigma)$, we will use the same terminology of ``medium'' and ``large'' disks in $Y(\sigma)$. There will not be any ``small'' disks in $Y(\sigma)$ as those disks are collapsed to points.

    \vskip 5pt

    \noindent {\bf Fact 2:} If $\hat D$ is a medium disk in $Y(\sigma)$, then it has overlap ratio $o(\hat D) < \lambda$.

    \vskip 5pt

    The disk $\hat D$ is the image of a medium disk $D$ in $X(\sigma)$. Since the index of $D$ is greater than the critical index $I$, we have $o(D)<\lambda$. From property (3), the disk $D$ is disjoint from all the small disks. So the quotient map $q$ leaves $D$ and all edges incident to the curve $\partial D$ unchanged. It follows that $o(\hat D)=o(D) <\lambda$, as desired.

    This leaves us with checking the overlap ratio of large disks in $Y(\sigma)$. In order to do this, we need to give a lower bound on the length of the large disk, and an upper bound on the length of the overlaps in the large disk. As before, we let $\hat D$ be a large disk in $Y(\sigma)$, which is the image of a large disk $D$ in  $X(\sigma)$. The boundary $\partial \hat D$ is obtained from $\partial D$ by collapsing the subpaths that are images of overlaps with small disks. 
    
    Note that, since $K> \mathcal O$, the length of $\partial D$ exceeds the length of any of the overlaps in $X(\sigma)$. So if $({\bf p}, {\bf p'})$ is an overlap with ${\bf p} \subset \widetilde{\partial D}$, we can instead view ${\bf p}$ as an embedded path in $\partial D$. We know from Corollary \ref{cor-bound-on-number-of-overlaps} that there are only finitely many overlaps in $X(\sigma)$, so we can list out all the overlaps $({\bf p}, {\bf p'})$ between $\hat D$ (so ${\bf p} \subset \widetilde{\partial D}$) and small disks. This gives us a finite list of overlaps $\{ ({\bf p_1}, {\bf p_1'}), \ldots ({\bf p_k}, {\bf p_k'})\}$, cyclically ordered according to the initial vertex of the paths ${\bf p_i} \subset \partial D$. Each ${\bf p_i'}$ lies in $\widetilde{\partial D_i}$ where $D_i$ is a small disk in $X(\sigma)$. The boundary $\partial \hat D$ is obtained from $\partial D$ by collapsing the images of the intervals ${\bf p_i} $ in  $\partial D$ to points. 

    \vskip 5pt

    \noindent {\bf Fact 3:} For any large disk $D$, there are at most $\leq M^2I(R_L\mathcal O)^2$
    many overlaps with small disks, i.e. $({\bf p}, {\bf p'})$, where ${\bf p} \subset \widetilde{\partial D}$ and ${\bf p'}$ is contained in any of the small disks.

    \vskip 5pt

    Applying Lemma~\ref{lem-lifting-overlaps}, any such overlap covers an overlap $({\bf q}, {\bf q'})$ in $X$. There are at most $M$ small disks in $X(\sigma)$, hence at most $M$ images of small disks in $X$. Since each image disk $E$ in $X$ has at most $R_L\mathcal O$ paths that could serve as ${\bf q'}$, the total number of possible image overlaps $({\bf q}, {\bf q'})$ in $X$ is bounded above by $M(R_L\mathcal O)^2$. 
    
    Lastly, given a candidate image overlap $({\bf q}, {\bf q'})$ in $X$, we need to check how many pre-images in $X(\sigma)$ correspond to overlaps between $D$ and one of the small disks. The image of the common path ${\bf q'}$ in $X^{(1)}$ lifts to $n$ paths inside $X(\sigma)^{(1)}$, each of them lying on some lift of the disk $E$. However, we know that there are at most $\leq M$ lifts that are small disks, and as each of them have index $\leq I$, there are $\leq MI$ lifts of ${\bf q'}$ along small disks. Since there is a bijection between the lifts of ${\bf q'}$ and those of ${\bf q}$ (they define the same path in $X^{(1)}$), we see that the pair $({\bf q}, {\bf q'})$ has at most $\leq MI$ lifts that are overlaps between the given disk $D$ and one of the short disk lifts of $E$. Combining this with the estimate on the number of possible projected overlaps in $X$ from the previous paragraph, {\bf Fact 3} follows.

    From the upper bound on the number of overlaps, we can deduce a lower bound on the length of the boundary $\partial \hat D$ for the quotient disk. The other ingredient we will need is to compute an upper bound on the size of the overlaps for the quotient disk $\hat D$. We have:

    \vskip 5pt
    
    \noindent {\bf Fact 4:} Any overlap for $\hat D$ has length $\leq \big(M^2I(R_L\mathcal O)^2+1\big)\mathcal O$.

    \vskip 5pt

    To see this, let us consider an overlap between $\hat D$ and some other disk $\hat E$ in the quotient space $Y(\sigma)$. These are images of disks $D, E$ inside $X(\sigma)$, and we would like to relate the overlaps in $X(\sigma)$ between $D, E$ with those in $Y(\sigma)$ between $\hat D, \hat E$. Observe that if we have an overlap in $X(\sigma)$ with the property that the two edges preceding and following survive in the quotient space, then the image will be an overlap in $Y(\sigma)$. But if the preceding and/or following edges are in the subsets being collapsed, then we can potentially lose the ``witness'' to the start/end of the overlap. In that case, in the quotient space the overlap could continue, as the subsets where they differed can be collapsed down to points. This would result in a potentially longer overlap in $Y(\sigma)$, obtained by concatenating two overlaps in $X(\sigma)$ (see Figure \ref{fig-overlaps-concatenate}).

    Now the only way such a concatentation can occur is if the overlap in $X$ started and ended on part of a short disk. More precisely, along the disk $D$ we have a collection of paths ${\bf p_i} \subset \partial D$ that come from overlaps with small disks. From {\bf Fact 3} there are at most $\leq M^2I(R_L\mathcal O)^2$ such overlaps. So at most $\leq M^2I(R_L\mathcal O)^2+1$ concatenations can occur. Since overlaps in $X(\sigma)$ have length at most $\mathcal O$, {\bf Fact 4} follows.

    \begin{figure}[h!]
    \includegraphics[width=9cm,height=3.5cm]{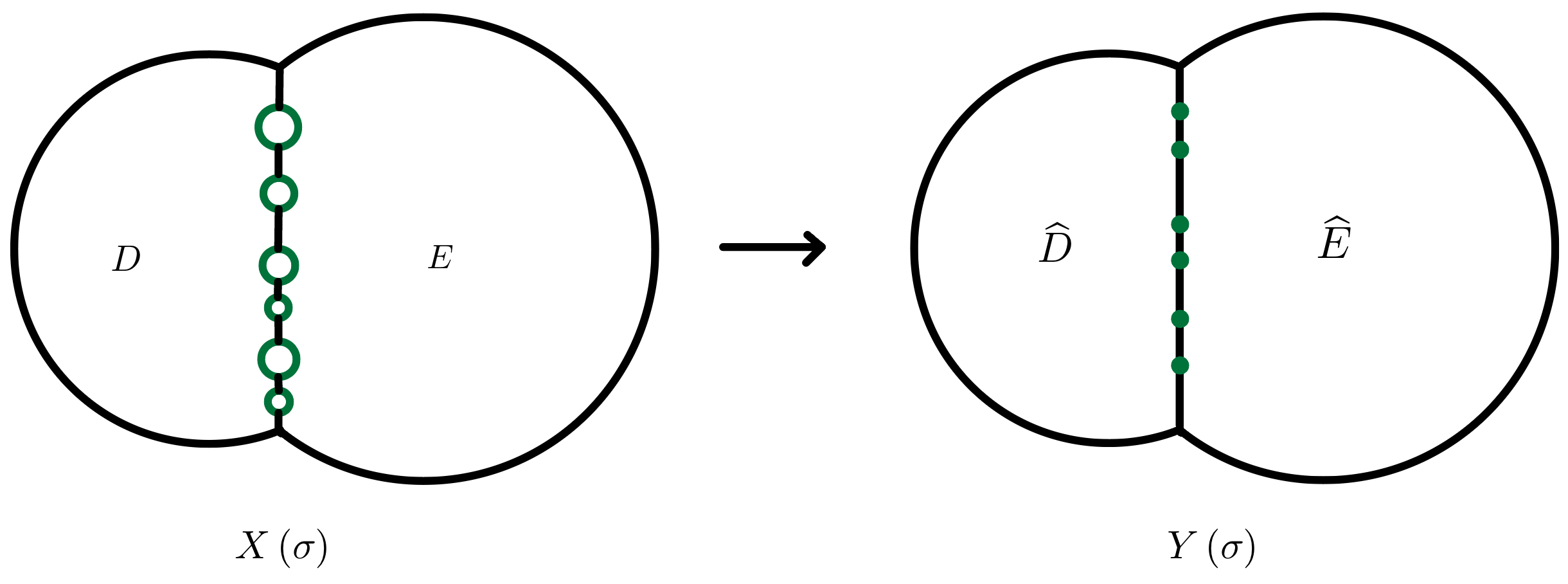}
    \caption{Collapsing small disks can concatenate overlaps in $X(\sigma)$ to give longer overlaps in $Y(\sigma)$}\label{fig-overlaps-concatenate}
    \end{figure}

    \vskip 5pt

    Finally, with {\bf Fact 3} and {\bf Fact 4} in hand, it is straightforward to estimate the overlap ratio of $\hat D$. Indeed, $\hat D$ is the image of the large disk $D$ in $X(\sigma)$ under the quotienting map. Since $D$ is a large disk in $X(\sigma)$, it is a branched cover of a disk in $X$, with index $\geq K$. So the length of $D$ is bounded below by $\geq KR_S$. When computing the length of $\hat D$, we see from {\bf Fact 3} that at most $\leq M^2I(R_L\mathcal O)^2$ overlaps with short disks get collapsed to points. Since each of these overlaps has length $\leq \mathcal O$, we get the lower bound
    \begin{align*}
            \ell (\partial \hat D) &\geq KR_S - M^2I(R_L\mathcal O)^2\mathcal O\\
            &\geq \lambda^{-1}\mathcal OM^2I(R_L\mathcal O)^2 + \lambda^{-1}\mathcal O
    \end{align*}
    where the second inequality follows from the chosen value of $K$, see Equation (\ref{eqn-K}). Now using the estimate for the overlap length from {\bf Fact 4} we get:
    $$o(\hat D) \leq \frac{\big(M^2I(R_L\mathcal O)^2+1\big)\mathcal O}{\lambda^{-1}\mathcal OM^2I(R_L\mathcal O)^2 + \lambda^{-1}\mathcal O} =\lambda$$
    as desired. Since this estimate holds for any large disk in our $X(\sigma)$ satisfying conditions (1), (2), (3), it concludes the proof of the theorem.
    
\end{proof}

\section{Concluding Remarks}

We end our paper with some general remarks on topics related to our random model and our main theorem.

\vskip 10pt

\subsection{Multiple vertex case}\label{sec:multiple-vertex-case}

The attentive reader will notice that our {\bf Main Theorem} is stated for acceptable $2$-complexes, but that in our proofs we work exclusively with the special case of an acceptable {\it presentation} $2$-complex. In fact, the two statements are equivalent, as we now explain.

Given an arbitrary finite acceptable $2$-complex $X$, we can take a spanning tree $T$ in the $1$-skeleton of $X$, and create a new $2$-complex $Z$ by collapsing $T$ to a point. By construction, the $1$-skeleton $Z^{(1)}$ is a bouquet of circles, so $Z$ has a single vertex. The quotient map $\phi: X\rightarrow Z$ is a homotopy equivalence, since it is obtained by collapsing the contractable set $T$. Each polygonal $2$-cell in $X$ gives rise to a polygonal $2$-cell in $Z$. Moreover, the restriction of $\phi$ to the $1$-skeleton is a homotopy equivalence between $X^{(1)}$ and $Z^{(1)}$, so provides an isomorphism $\phi_\#: \pi_1(X^{(1)}) \rightarrow \pi_1(Z^{(1)})$ . It follows that an attaching map $\alpha$ for a disk in $X$ is a proper power in $\pi_1(X^{(1)})$ if and only if the corresponding attaching map $\phi\circ \alpha$ for a disk in $Z$ is a proper power. Similarly a pair of disks have identical attaching map in $\pi_1(X^{(1)})$ if and only if the corresponding attaching maps in $\pi_1(Z^{(1)})$ are identical. This shows that $Z$ is also an acceptable $2$-complex, but with a single vertex, so can be viewed as a presentation $2$-complex. 

Finally, our model for random branched covers of $Z$ are obtained by taking ordinary degree $n$ covers of the $1$-skeleton $Z^{(1)}$, and inducing a branched cover by attaching disks along all the connected lifts of an attaching map (see Section \ref{subsec:brcov} and Section \ref{subsec:random-model}). From covering space theory, all information on lifting is encoded in the fundamental group of the $1$-skeletons. Hence the group isomorphism $\phi_\#$ allows you to obtain a corresponding finite cover of $X^{(1)}$, and a homotopy equivalence between this finite cover and the $1$-skeleton of the branched cover $Z(\sigma)$. Under this homotopy equivalence, we can transfer the lifts of the attaching maps to the finite cover of $X^{(1)}$ and form a corresponding branched cover $X(\sigma)$ of $X$. By construction, there is then an induced homotopy equivalence $X(\sigma) \simeq Z(\sigma)$. It follows that topological results about the random model can be transferred from the acceptable presentation $2$-complex case to the general case of acceptable polygonal $2$-complexes.

\vskip 10pt
\subsection{Non-acceptable $2$-complexes}\label{subsec:non-acceptable}

In Section \ref{subsec:random-model}, we discussed a possible extension of our random model to general $2$-complexes (not necessarily acceptable). To understand the topology of a random branched cover in the non-acceptable case, the following observation will be useful.

\begin{lemma}\label{lem:pi2}
    Let $X$ be a $2$-complex, where a pair of $2$-cells are attached along homotopic maps $\alpha \simeq \beta : S^1\rightarrow X^{(1)}$. Then $\pi_2(X)\neq 0$, so $X$ is not aspherical.
\end{lemma}

\begin{proof}
    Up to homotopy equivalence, we can assume that $\alpha = \beta$, and that these are the last two disks attached. Denote by $X_\beta$ the space obtained from $X$ by removing the disk attached along $\beta$. Then the attaching map $\beta$, viewed as a map into $X_\alpha$, can be homotoped across the disk attached along $\alpha$ to a point map. It follows that $X$ is homotopy equivalent to the join $X_\beta \vee S^2$. Applying the $\pi_2$-functor, and noting that $S^2$ is a retract of $X_\beta \vee S^2$, we see that 
    $$\mathbb Z \cong \pi_2(S^2) \hookrightarrow \pi_2(X_\beta \vee S^2)\cong \pi_2(X)$$
    establishing the result.
\end{proof}    

For the random model with multiplicity, any disk whose attaching map is a proper power $\alpha ^k$ ($\alpha$ not a proper power, and $k\geq 2$) can give rise to multiple disks in $X(\sigma)$ attached along the same lifted map $\tilde \alpha$. This occurs anytime the degree $d$ of the lifted map $\tilde \alpha \rightarrow \alpha$ has a divisor in common with the power $k$. In those situations, Lemma \ref{lem:pi2} applies and guarantees that $X(\sigma)$ is not aspherical. Note that the degrees $d$ of the lifted map are the lengths of the cycles in the image of the word map $\alpha(\sigma)$. Since $\alpha$ is not a proper power, the permutations $\alpha(\sigma)$ are uniformly distributed in the symmetric group as $\sigma$ ranges over all tuples of permutations. In particular, for $n$ sufficiently large, there is a positive proportion of $\sigma$ that will produce a permutation $\alpha(\sigma)$ containing a $k$-cycle. This gives us the following

\begin{cor}
    Assume the $2$-complex $X$ has a disk where the attaching map is a proper power $\alpha^k$ in $\pi_1(X^{(1)})$. Then for $n$ large, a positive proportion of the random branched covers $X(\sigma)$ in the random model with multiplicity will satisfy $\pi_2(X(\sigma)) \neq 0$.
\end{cor}

\vskip 10 pt

\subsection{Effect on the fundamental group}

\begin{lemma}
    If $\rho: \bar X \rightarrow X$ is a branched covering map between finite, connected $2$-complexes, then the induced map $\rho_\sharp$ maps $\pi_1(\bar X)$ onto a finite index subgroup of $\pi_1(X)$.
\end{lemma}

\begin{proof}
    We know that the inclusion of the $1$-skeleton $\iota:X^{(1)}\hookrightarrow X$ induces a surjection from a finite rank free group $\iota_\sharp:\pi_1(X^{(1)}) \twoheadrightarrow \pi_1(X)$. The branched covering map $\rho$ restricts to a covering map on the $1$-skeletons, so we have a well-defined finite index subgroup $\rho_{\sharp}\big(\pi_1(\bar X^{(1)})\big) \leq \pi_1(X^{(1)})$. The $\iota_\sharp$-image of this subgroup will be finite index in $\pi_1(X)$, and will be contained in the $\rho_\sharp\big(\pi_1(\bar X)\big)$, completing the proof.
\end{proof}

In the special case where $X$ is a presentation $2$-complex associated to a finitely presented group $\Gamma$, we see that $\pi_1(\bar X)$ surjects onto a finite index subgroup of $\Gamma$. We do not expect, in general, to have good finiteness properties on the kernel of this map, in contrast with e.g. the Rips construction. Nevertheless, one can wonder whether geometric or topological properties of the group $\Gamma$ can be inherited by the group $\pi_1(\bar X)$. For example, Corollary \ref{lemma:sc} tells us that if the presentation $2$-complex $X$ satisfies geometric small cancellation, then so does the branched cover $\bar X$.

\vskip 10pt

\subsection{Non-uniform measures}

In our model for random branched covers, we always use the sequence of uniform measures on the symmetric groups $\Sym(n)$. It is reasonable to consider the branched covers associated to other sequences of measures on the symmetric groups, and to see if one can obtain other types of prescribed behavior for the corresponding random branched covers. For example, if you start with an $X$ that is {\bf not} small cancellation, can one bias the measures to guarantee that a random branched cover remains asymptotically almost surely not small cancellation? Or would that force the sequence of measures to concentrate support on some small set of permutations (e.g. looking more and more like Dirac measures)? Since the attaching maps in the branched covers are determined by word maps associated to the relators, these questions will likely require finer understanding of how word maps interact with measures on the symmetric group.

\bigskip

\bibliographystyle{alpha}
\bibliography{Bibliography}

\end{document}